\documentclass{amsart}
\usepackage{mathrsfs}
\usepackage{amsfonts,amssymb,amsmath,amsthm}
\usepackage{url}
\usepackage{enumerate}
\usepackage{float}
\usepackage{subfig, epsfig}
\usepackage{graphicx}

\newtheorem{thm}{Theorem}[section]
\newtheorem{lem}[thm]{Lemma}
\newtheorem{prop}[thm]{Proposition}

\newtheorem{remark}[thm]{Remark}
\numberwithin{equation}{section}

\newcommand{\R}{\mathbb{R}}

\allowdisplaybreaks

\author{Ningkui Sun}
\address{School of Mathematics and Statistics\\
Shandong Normal University\\
 Jinan 250014, China}
 \email{sunnk1987@163.com}

\author{Jian Fang}
\address{
Institute for Advanced Studies in Mathematics and Department of Mathematics\\
 Harbin Institute of Technology\\
 Harbin 150001, China}
 \email{jfang@hit.edu.cn}


\thanks{The research leading to these results was supported by the NSF of China (No.11771108),
 the NSF of Heilongjiang province of China (No. LC2017002) and the NSF of Shandong Province of China (No. ZR201702140024).}

\begin{document}

\title[Fisher-KPP equation with time delay]
{Propagation dynamics of Fisher-KPP equation with time delay and free boundaries}

\begin{abstract}
Incorporating free boundary into time-delayed reaction-diffusion equations yields a compatible condition that guarantees  the well-posedness of the initial value problem. With the KPP type nonlinearity we then establish a vanishing-spreading dichotomy result. Further, when the spreading happens, we show that the spreading speed and spreading profile are nonlinearly determined by a delay-induced nonlocal semi-wave problem. It turns out that time delay slows down the spreading speed.
\end{abstract}

\maketitle

\section{Introduction}\label{sec:intr}
In the pioneer work of Fisher \cite{Fisher}, and  Kolmogorov, Petrovski and Piskunov \cite{KPP}, it was shown that
\begin{equation}\label{eq:KPP}
u_t =u_{xx}+f(u), \quad x\in\R
\end{equation}
with
\begin{equation}\label{eq:KPP-cond}
f\in C^1(\R,\R),\quad f(0)=0=f(1),\quad f(s)\leqslant f'(0)s,\ s\geqslant 0,
\end{equation}
admits traveling waves solutions of the form $u(t,x)=\phi(x-ct)$ satisfying $\phi(-\infty)=1$ and $\phi(+\infty)=0$ if and only if $c\geqslant c_0:=2\sqrt{f'(0)}$. In 1970s', Aronson and Weinberger \cite{AW2} proved that the minimal wave speed $c_0$ is also the asymptotic speed of
spread (spreading speed for short) in the sense that
\begin{equation}
\lim_{t\to\infty} \sup_{|x|\geqslant (c_0+\epsilon)t}u(t,x)=0,\quad \lim_{t\to\infty} \inf_{|x|\leqslant (c_0-\epsilon)t}u(t,x)=1
\end{equation}
for any small $\epsilon>0$ provided that the initial function $u(0,x)$ is compactly supported. These works have stimulated volumes of studies
for the propagation dynamics of various types of evolution systems. Among others, of particular interest to the Fisher-KPP equation
\eqref{eq:KPP}-\eqref{eq:KPP-cond} with time delay or free boundary are two typical ones.

Schaaf \cite{Sc} studied the following delayed reaction-diffusion equation
\begin{equation}\label{eq:Schaaf}
u_t(t,x)=u_{xx}(t,x)+f(u(t,x),u(t-\tau,x)),\quad x\in\R,\ t>0,
\end{equation}
where $\tau>0$ is the time delay. With the Fisher-KPP condition on $\tilde{f}(s):=f(s,s)$ and  the quasi-monotone condition $\partial_2 f\geqslant 0$,
it was shown that the minimal wave speed $c_0=c_0(\tau)$ exists and it is determined by the system of two transcendental equations
\begin{equation}\label{eq:speed}
F(c,\lambda)=0, \quad \frac{\partial F}{\lambda}(c,\lambda)=0,
\end{equation}
where
\begin{equation}
F(c,\lambda)=\lambda^2+c\lambda+\partial_1 f(0,0)+\partial_2 f(0,0)e^{-\lambda\tau}.
\end{equation}
The delay-induced spatial non-locality was brought to attention by So, Wu and Zou \cite{SWZ}, where they derived the following time-delayed reaction-diffusion model equation with nonlocal response for the study of age-structured population
\begin{equation}\label{eq:SWZ}
u_t=u_{xx}-d u+\gamma \int_\R b(u(t-\tau,x-y))k(y)dy,\quad x\in\R,\ t>0,
\end{equation}
where $u$ represents the density of mature population, $\tau>0$ is the maturation age, $d$ is the death rate, $b$ is the birth rate function, $\gamma$ is the survival rate from newborn to being mature, and $k$ is the redistribution kernel during the maturation period. As such,  introducing time delay into diffusive equation usually gives rises to spatial non-locality due to the interaction of time lag (for maturation) and diffusion of immature population. In the extreme case where the immature population does not diffuse, the kernel $k$ becomes the Dirac measure, and hence \eqref{eq:SWZ} reduces to \eqref{eq:Schaaf}. We refer to the survey article \cite{GW} for the delay-induced nonlocal reaction-diffusion problems. In \cite{SWZ}, the authors obtained the minimal wave speed $c_0(\tau)$ that is determined by a similar system to \eqref{eq:speed} provided that $b$ is nondecreasing and $f(s):=-ds+b(s)$ is of Fisher-KPP type. Wang, Li and Ruan \cite{WLR} proved that $c_0(\tau)$ is decreasing in $\tau$. Liang and Zhao \cite{LZ} showed that $c_0(\tau)$ is also the spreading speed for the solutions satisfying the following initial condition
\begin{equation}
\text{$u(\theta,x)$ is continuous and compactly supported in $\theta\in [-\tau,0]$ and $x\in\R$.}
\end{equation}
Similar to the classical Fisher-KPP equation, the spreading speed $c_0(\tau)$ for delayed reaction-diffusion equation is still linearly determined for both local and nonlocal problems thanks to the Fisher-KPP type condition.

We refer to \cite{MS} for more properties that are induced by time delay in reaction-diffusion equations, including the well-posedness of initial value problems as well as the role of the quasi-monotone condition on the comparison principle, and \cite{FangZhao14, FangZhao15} for the delay-induced weak compactness of time-$t$ solution maps when $t\in(0,\tau]$ as well as its role in the study of wave propagation.

Recently, Du and Lin \cite{DuLin} proposed a Stefan type free boundary to the Fisher-KPP equation
\begin{equation}\label{freeb}
\left\{
\begin{array}{ll}
 u_t =u_{xx}+u(1-u), &  g(t)< x<h(t),\; t>0,\\
 u(t,g(t))=0,\ \ g'(t)=-\mu u_x(t, g(t)), & t>0,\\
 u(t,h(t))=0,\ \ h'(t)=-\mu  u_x  (t, h(t)) , & t>0,
\end{array}
\right.
\end{equation}
where the free boundaries $x=g(t)$ and $x=h(t)$
represent the spreading fronts, which are determined jointly  by the gradient at the fronts and the coefficient $\mu$ in the Stefan condition. For more background of proposing such free boundary conditions, we refer to \cite{DuLin, BDK}. It was proved in \cite{DuLin} that
the unique global solution $(u,g,h)$ has a spreading-vanishing dichotomy property as $t\to\infty$: either $(g(t),h(t))\to\R$ and $u\to1$ (spreading case), or $g(t)\to g_\infty$, $h(t)\to h_\infty$ with $h_\infty-g_\infty\leqslant \pi$, and $u\to 0$ (vanishing case). Moreover,
it was also proved that when spreading happens, there is a constant $k_0>0$ such that $-g(t)$ and $h(t)$ behave like a straight line $k_0t$ for large time, where $k_0$ is called the asymptotic speed of spread (spreading speed for short). Different from the classical Fisher-KPP speed, $k_0$ is the unique value of $c$ such that the following nonlinear semi-line problem is solvable:
\begin{equation}\label{k0}
\left\{
 \begin{array}{ll}
 q'' - cq'+q(1-q)=0, & z>0,\\
 q(\infty)=1, \ \ \mu q_+'(0)=c,\ \ q(z)>0, & z\leqslant 0,\\
 q(z)=0,  & z\leqslant 0,
 \end{array}
 \right.
\end{equation}
where $q_+'(0)$ is the right derivative of $q(z)$ at $0$.
In particular, as $\mu$ increases to infinity, $k_0$ increases to the classical Fisher-KPP speed $2\sqrt{f'(0)}$.
Later on, Du and Lou \cite{DuLou} obtained a rather complete characterization on the asymptotic behavior of solutions for \eqref{freeb} with
some general nonlinear terms. For further related work on free boundary problems,
we refer to \cite{DuGuo, DGP, DMZ} and the references therein.

In this paper, we aim to explore how to incorporate time delay and free boundary into the Fisher-KPP equation \eqref{eq:KPP}-\eqref{eq:KPP-cond} so that the problem is well-posed, and then study their joint influence on the propagation dynamics.

Keeping a smooth flow for the organizations of the paper, we write down here the problem of interest while leaving in the next section  the derivation details, including the emergence of the compatible condition \eqref{CC} for the well-posedness of the initial value problem.
\begin{equation}\label{p}
\left\{
\begin{array}{ll}
 u_t(t,x) =u_{xx}(t,x)- d u(t,x) +f(u(t-\tau,x)), &  x\in(g(t),h(t)),\; t>0,\\
 u(t,g(t))=0,\ \ g'(t)=-\mu u_x(t, g(t)), & t>0,\\
 u(t,h(t))=0,\ \ h'(t)=-\mu  u_x  (t, h(t)) , & t>0,\\
 u(\theta,x) =\phi (\theta,x),& g(\theta) \leqslant x \leqslant h(\theta),\; \theta\in[-\tau,0],
\end{array}
\right.
\tag{$P$}
\end{equation}
where $d$ and $\tau$ are two positive constants, the nonlinear function $f$ satisfies
\[
\bf{(H)}\  \hskip 16mm
\left\{
 \begin{array}{l}
  f(s)\in C^{1+\tilde{\nu}}([0,\infty))\ \mbox{ for some } \tilde{\nu}\in(0,1),\ \ f(0)=0,\ \ f'(0)-d>0;\\
 f(s)-d s=0 \mbox{ has a unique positive constant root } u^*;\\
   f(s) \mbox{ is monotonically increasing in } s \in[0,u^*];\\
   \frac{f(s)}{s} \mbox{ is monotonically decreasing in } s \in[0,u^*]
  \end{array}
 \right. \ \hskip 15mm \hfill
\]
and the initial data $(\phi(\theta,x), g(\theta), h(\theta))$ satisfies
\begin{equation}\label{def:X}
\left\{
\begin{array}{ll}
\phi(\theta,x) \in C^{1,2} ([-\tau,0]\times[g(\theta),h(\theta)]),\\
0<\phi(\theta,x)\leqslant u^*\ \mbox{ for } (\theta,x)\in[-\tau,0]\times(g(\theta), h(\theta)),\\
\phi(\theta,x) \equiv 0\ \ \mbox{ for } \theta\in[-\tau,0],\; x\not\in(g(\theta),h(\theta))
\end{array}
\right.
\end{equation}
as well as the compatible condition
\begin{equation}\label{CC}
[g(\theta), h(\theta)]\subset [g(0), h(0)]\ \ \ \mbox{ for }\ \theta\in[-\tau,0].
\end{equation}

Assumption {\bf (H)} ensures the Fisher-KPP structure as well as the comparison principle. Due to the nature of delay differential equations,
the initial value, including the initial domain, has to be imposed over the history period $[-\tau,0]$, as in \eqref{def:X}. The interaction of time delay and free boundary gives rise to the compatible condition \eqref{CC} that is essential for the well-posedness of the problem. If $\tau=0$, then the compatible condition (1.12) becomes trivial and problem (P) reduces to \eqref{freeb}.

\begin{thm}\label{wellposedness}
{\rm(Well-posedness)} For an initial data $(\phi(\theta,x), g(\theta), h(\theta))$ satisfying \eqref{def:X} and \eqref{CC},  there exists a
unique triple $(u, g, h)$ solving \eqref{p} with $u\in C^{1,2}((0,\infty) \times[g(t),h(t)])$ and $g,\, h\in C^1([0,\infty))$.
\end{thm}
With the compatible condition  \eqref{CC} we can cast the problem into a fixed boundary problem and then apply the Schauder fixed point
theorem to establish the local existence of solutions. The extension to all positive time is based on some a priori estimates\footnote{We
sincerely thank Professor Avner Friedman for his valuable comments and suggestions on the proof of the well-posedness.}.

From the maximum principle and {\bf (H)}, it follows that when $t>0$ the solution $u>0$ as $x\in (g(t),h(t))$, $u_x(t,g(t))>0$ and $u_x(t,h(t))<0$, and hence, $g'(t)<0<h'(t)$ for all $t>0$. Therefore, we can denote
$$
g_{\infty}:=\lim_{t\to\infty}g(t)\ \ \ \mbox{and }\quad h_{\infty}:= \lim_{t\to\infty}h(t).
$$

\begin{thm}{\rm (Spreading-vanishing dichotomy)}\label{thm:asy be}
Let $(u,g,h)$ be the solution of \eqref{p}. Then the following alternative holds:

Either

{\rm (i) Spreading:} $(g_\infty, h_\infty)=\R$ and
\[
\lim_{t\to\infty}u(t,x)=u^* \mbox{ locally uniformly in $\R$},
\]

or

{\rm (ii) Vanishing:} $(g_\infty, h_\infty)$ is a finite interval
with length no bigger than $\frac{\pi}{\sqrt{f'(0)-d}}$ and
\[
\lim_{t\to\infty}\max_{g(t)\leqslant x\leqslant h(t)} u(t,x)=0.
\]
\end{thm}

\medskip

When spreading happens, we characterize the spreading speed and profile of the solutions. The nonlinear and nonlocal semi-wave problem
\begin{equation}\label{sw11}
\left\{
 \begin{array}{ll}
 q'' - cq'-d q+ f( q(z-c\tau))=0, & z>0,\\
 q(\infty)=u^*, \ \ \mu q_+'(0)=c,\ \ q(z)>0, & z\leqslant 0,\\
 q(z)=0,  & z\leqslant 0
 \end{array}
 \right.
\end{equation}
will play an important role. If $\tau=0$ then \eqref{sw11} reduces to the local form \eqref{k0}.

\begin{thm}\label{thm:semiwave}
Problem \eqref{sw11} admits a unique solution $(c^*, q_{c^*})$ and $c^*=c^*(\tau)$ is decreasing in delay $\tau\geqslant 0$.
\end{thm}
Due to the presence of time delay, the proof of Theorem \ref{thm:semiwave} highly relies on the distribution of complex
solutions of the following transcendental equation
\begin{equation}
\lambda^2-c\lambda-d+f'(0)e^{-\lambda c\tau}=0.
\end{equation}
We refer to Lemma \ref {lem:eigen} and Proposition \ref{prop:qoan1}, which are independently of interest.

With the semi-wave established above, we can construct various super- and subsolutions to estimate the spreading
fronts $h(t),g(t)$ and the spreading profile as $t\to\infty$.

\begin{thm}\label{thm:profile of spreading sol}
{\rm(Spreading profile)}
Let $u$ be a solution satisfying Theorem \ref{thm:asy be}(i).  Then there exist two constants $H_1$ and  $G_1$ such that
\[
\lim\limits_{t\to\infty}[h(t)- c^*t] = H_1 ,\quad \ \lim\limits_{t\to\infty} h'(t)=c^*,
\]
\[
\lim\limits_{t\to\infty}[g(t) + c^*t] = G_1 ,\quad\ \lim\limits_{t\to\infty} g'(t)=-c^*,
\]
\begin{equation}\label{profile convergence 1}
\lim\limits_{t\to\infty} \left\| u(t,\cdot)- q_{c^*}(c^*t+ H_1-\cdot) \right\|_{L^\infty ( [0, h(t)])}=0,
\end{equation}
\begin{equation}\label{profile convergence 1-left}
\lim\limits_{t\to\infty} \left\| u(t,\cdot)- q_{c^*}(c^*t- G_1+\cdot )\right\|_{L^\infty ([g(t), 0])} =0,
\end{equation}
where $(c^*,q_{c^*})$ is the unique solution of \eqref{sw11}.
\end{thm}

The rest of this paper is organized as follows. In Section 2 we derive the compatible condition \eqref{CC}, with which we formulate problem (P) and then establish the well-posedness as well as the comparison principle. Section 3 is devoted to the study of the semi-wave problem \eqref{sw11}. In section 4, we establish the spreading-vanishing dichotomy result. Finally in Section 5, we characterize the spreading speed and profile of spreading solutions of \eqref{p}.

\section{The compatible condition, well-posedness and comparison principle}\label{sec:basic}

\subsection{The compatible condition}
To formulate problem \eqref{p}, we start from the age-structured population growth law
\begin{equation}\label{ase}
p_t+p_a=D(a)p_{xx}-d(a)p,
\end{equation}
where $p=p(t,x;a)$ denotes the density of species of age $a$ at time $t$ and location $x$, $D(a)$ and $d(a)$ denote the diffusion rate
and death rate of species of age $a$, respectively.

Next we consider the scenario that the species has the following biological characteristics.
\begin{itemize}
\item[(A1)] The species can be classified into two stages by age: mature and immature. An individual at time
$t$ belongs to the mature class if and only if its age exceeds the maturation time $\tau>0$. Within each stage,
all individuals share the same behavior.
\item[(A2)] Immature population does not move in space.
\end{itemize}
The total mature population $u$ at time $t$ and location $x$ can be represented by the integral
\begin{equation}\label{mimuv}
u(t,x)=\int_\tau^\infty p(t,x;a)da.
\end{equation}
We assume that the mature population $u$ lives in the habitat $[g(t),h(t)]$,  vanishes in the boundary
\begin{equation}\label{vc}
u(t,g(t))=0=u(t,h(t)),\quad t>0
\end{equation}
and extends the habitat by obeying the Stefan type moving boundary conditions:
\begin{equation}\label{fbc}
h'(t)=-\mu u_x(t,h(t)),\ \ g'(t)=-\mu u_x(t,g(t)), \quad t>0,
\end{equation}
where $\mu$ is a given positive constant.  Note that the immature population does not contribute to the extension of habitat due to their immobility, as assumed in (A2).

According to (A1) we may assume that
\[
D(a)=\left\{
\begin{array}{ll}
1,& a\geqslant \tau ,\\
0, & 0\leqslant a<\tau,
\end{array} \right.
\ \ \ \
d(a)=\left\{
\begin{array}{ll}
d ,& a\geqslant \tau ,\\
d_I, & 0\leqslant a<\tau,
\end{array} \right.
 \]
 where $d$ and $d_I$ are two positive constants.  Differentiating the both sides of the equation \eqref{mimuv}
in time yields
\begin{eqnarray}\label{diff-u}
u_t & = &\int_\tau^\infty p_t da = \int_\tau^\infty [-p_a+ p_{xx}-d p]da\nonumber\\
& = & u_{xx}-d u+p(t,x;\tau) -p(t,x;\infty).\label{du1}
\end{eqnarray}
Since no individual lives forever, it is nature to assume that
\begin{equation}\label{infinity}
p(t,x;\infty)=0.
\end{equation}
To obtain a closed form of the model, one then needs to express $p(t,x;\tau)$ by $u$ in a certain way.
Indeed, $p(t,x;\tau)$ denotes the newly matured population at time $t$, and it is the evolution result of newborns at $t-\tau$.
In other words, there is an evolution relation between the quantities $p(t,x;\tau)$ and $p(t-\tau,x;0)$. Such a relation is obeyed by the
growth law \eqref{ase} for $0<a<\tau$, and hence it is the time-$\tau$ solution map of the following equation
\begin{equation}\label{ast}
\left\{
 \begin{array}{ll}
 q_s=-d_Iq, & x\in\R,\ 0\leqslant s\leqslant \tau,\\
 q(0,x)=p(t-\tau,x;0),  & x\in\R.
 \end{array}
 \right.
\end{equation}
Therefore, $p(t,x;\tau)=q(\tau,x)=e^{-d_I\tau}p(t-\tau,x,0)$. Further, the newborns $p(t-\tau,x;0)$ is given by the birth $b(u(t-\tau,x))$, where $b$ is the birth rate function with $b(0)=0$. Consequently,
\begin{equation}\label{pptst}
p(t,x;\tau)=e^{-d_I\tau}b(u(t-\tau,x)).
\end{equation}
Combining \eqref{vc}-\eqref{infinity} and \eqref{pptst},  we are led to the following system:

\begin{equation}\label{p-nonsim}
\left\{
\begin{array}{lll}
u_t(t,x) =u_{xx}(t,x)- d u(t,x) +e^{-d_I\tau}b(u(t-\tau,x)), &  t>0, x\in[g(t-\tau),h(t-\tau)]\\
u_t(t,x) =u_{xx}(t,x)- d u(t,x), &  t>0, x\in[g(t),h(t)]\setminus [g(t-\tau),h(t-\tau)]\\
u(t,g(t))=0=u(t,h(t)), &t>0\\
h'(t)=-\mu u_x(t,h(t)),\ \ g'(t)=-\mu u_x(t,g(t)), & t>0.
\end{array} \right.
\end{equation}
For $t>0$, outside the habitat $(g(t),h(t))$ the mature population does not exist, that is,
\begin{equation}\label{uhhgg}
u(t,x)\equiv0 \ \ \ \mbox{ for } \ t>0,\; x\not\in(g(t),h(t)).
\end{equation}
Clearly, since the habitat is expanding for $t>0$, we have
\begin{equation}\label{habitat}
[g(t-\tau),h(t-\tau)]\subset [g(t),h(t)],\quad t\geqslant \tau.
\end{equation}
Hence, the first two equations in \eqref{p-nonsim} can be written as the following single one
\begin{equation}
u_t(t,x) =u_{xx}(t,x)- d u(t,x) +e^{-d_I\tau}b(u(t-\tau,x)), \quad  t>0, x\in[g(t),h(t)]
\end{equation}
provided that \eqref{habitat} holds for $t\geqslant 0$. As such, in view of \eqref{habitat} we need an additional condition
\begin{equation}\label{AC}
[g(t-\tau),h(t-\tau)]\subset [g(t),h(t)], \quad t\in[0,\tau).
\end{equation}
Note that $[g(0),h(0)]\subset[g(t),h(t)]$ for $t>0$. And as the coefficient $\mu\to+\infty$ we have $[g(t),h(t)]\to [g(0,h(0))]$ uniformly for $t\in [0,\tau]$. Therefore, regardless of the influence of $\mu$, \eqref{AC} is strengthened to be
\[
[g(\theta), h(\theta)]\subset [g(0), h(0)]\ \ \ \mbox{ for }\ \theta\in[-\tau,0],
\]
which is the aforementioned compatible condition \eqref{CC}.

Setting $f(s):=e^{-d_I\tau}b(s)$ in \eqref{p-nonsim}, we obtain problem \eqref{p}.

\subsection{Well-posedness}
We employ the Schauder fixed point theorem to establish the local existence of solutions to \eqref{p},  and prove the uniqueness,
 then extend the solutions to all time by an estimate on the free boundary.

\begin{thm}\label{thm:local}
Suppose {\bf(H)} holds. For any $\alpha\in (0,1)$, there is a $T>0$ such that problem
\eqref{p} admits a solution
$$(u, g, h)\in C^{(1+\alpha)/2,
1+\alpha}([0,T]\times[g(t),h(t)])\times C^{1+\alpha/2}([0,T])\times C^{1+\alpha/2}([0,T]).$$
\end{thm}

\begin{proof}  We divide the proof into three steps.

\smallskip

$Step\ 1$. We use a change of variable argument to transform problem \eqref{p} into a fixed boundary problem with
a more complicated equation which is used in \cite{CF, DuLin}. Denote $l_1=g(0)$ and $l_2=h(0)$ for convenience, and set
$h_0=\frac{1}{2}(l_2-l_1)$. Let $\xi_{1}(y)$ and $\xi_{2}(y)$ be two nonnegative functions in $C^{3}(\R)$ such that
\[
\xi_{1}(y)=1\  \mbox{ if}\  | y-l_2|< \frac{h_{0}}{4},\  \xi_{1}(y)=0\ \mbox{ if} \ |y-l_2| > \frac{h_{0}}{2},\
|\xi_{1}'(y)|<\frac{6}{h_{0}}\ \mbox{for}\  y\in \R;
\]
\[
\xi_{2}(y)=1\ \mbox{ if}\  | y-l_1| < \frac{h_{0}}{4},\  \xi_{2}(y)=0\ \mbox{ if}\  | y-l_1| > \frac{h_{0}}{2},\
|\xi_{2}'(y)| < \frac{6}{h_{0}}\ \mbox{for}\  y\in \R.
\]
Define $y= y(t,x)$ through the identity
\begin{align*}
&x=y+\xi_{1}(y)(h(t)-l_2)+\xi_{2}(y)(g(t)-l_1)\  \ \ \mbox{ for } t>0,\\
&x\equiv y\ \ \ \mbox{ for } -\tau\leqslant t\leqslant 0.
\end{align*}
and set
\begin{align*}
&w(t,y):=u(t,y+\xi_{1}(y)(h(t)-l_2)+\xi_{2}(y)(g(t)-l_1))=u(t,x)\  \ \ \mbox{ for } t>0,\\
&w(\theta,y):=\phi(\theta,y)\ \ \ \mbox{ for } -\tau\leqslant \theta\leqslant 0.
\end{align*}
Then the free boundary problem \eqref{p} becomes
\begin{equation}\label{lin1}
\left\{
\begin{array}{ll}
 w_t -A(g,h,y)w_{yy} + B(g,h,y)w_{y}=f(w(t-\tau,y))- d  w, &  y\in(l_1, l_2),\ t>0,\\
 w(t,l_i)=0,  &  t>0,\ i=1, 2,\\
w(\theta,y) =\phi(\theta,y),& y\in[l_1,l_2],\ \theta\in[-\tau,0],
\end{array}
\right.
\end{equation}
and
\begin{equation}\label{lghin1}
g'(t)=-\mu\, w_y(t,l_1), \ \ h'(t) = -\mu w_{y}(t,l_2),\ \ \ t>0,
\end{equation}
with $f(w(t-\tau,y))=f(u(t-\tau,y))$ and $A(g,h,y)=[1+\xi_1'(y)(h(t)-l_2)+\xi_2'(y)(g(t)-l_1)]^{-2}$,
\[
B(g,h,y)=[\xi_1''(y)(h(t)-l_2)+\xi_2''(y)(g(t)-l_1)]A(g,h,y)^{\frac {3}{2}}-[\xi_1(y)h'(t)+\xi_2(y)g'(t)]A(g,h,y)^{\frac {1}{2}}.
\]

Denote $ h_{1}=-\mu (u_{0})_y(0,l_2)$, and  $h_{2}=\mu (u_{0})_y(0,l_1)$. For $ 0<T\leqslant \min\big\{\frac{h_{0}}{16(1+ h_{1} +h_{2})},\ \tau\big\}$,
we define $\Omega_{T}:=[0,T]\times[l_1,l_2]$,
\begin{align*}
&\mathcal{D}^{h}_{T}=\{h\in C^{1}([0,T]):\ h(0)=l_2,\ h'(0)=h_{1},\ \| h'-h_{1}\|_{C([0,T])} \leqslant 1\},\\
&\mathcal{D}^{g}_{T}=\{g\in C^{1}([0,T]):\ g(0)=l_1,\ g'(0)=-h_{2},\ \| g'+h_{2}\|_{C([0,T])} \leqslant 1\}.
\end{align*}
Clearly, $\mathcal{D}:=\mathcal{D}^{g}_{T}\times\mathcal{D}^{h}_{T}$ is a bounded and closed convex set of $C^1([0,T])\times C^1([0,T])$.

Noting that the restriction on $T$, it is easy to see that the transformation $(t,y)\rightarrow(t,x)$ is well defined.
By a similar argument as in \cite{W}, applying standard $L^p$ theory and the Sobolev embedding theorem, we can deduce that for any given $(g,h)\in \mathcal{D}$,
problem \eqref{lin1} admits a unique $w(t,y;g,h)\in W^{1,2}_p(\Omega_{T})\hookrightarrow C^{\frac{1+\alpha}{2},{1+\alpha}}(\Omega_{T})$, which satisfies
\begin{equation}\label{eq1}
  \|w\| _{W^{1,2}_p(\Omega_{T})}+\|w\| _{C^{\frac{1+\alpha}{2},{1+\alpha}}(\Omega_{T})}\leqslant C_{1},
\end{equation}
 where $p>1$ and $C_{1}$ is a constant dependent on $g(\theta)$, $h(\theta)$, $\alpha$, $p$ and $\| \phi\|_{C^{1,2}([-\tau,0]\times[g(\theta),h(\theta)])}$.

Defining $\hat{h}$ and $\hat{g}$ by
$\hat{h}(t)=l_2-\int_0^t \mu w_{y}(s, l_2)ds$ and $\hat{g}(t)=l_1-\int_0^t \mu w_{y}(s, l_1)ds$, respectively,
then we have
\[
\hat{h}'(t)=-\mu w_{y}(t, l_2),\ \hat{h}(0)=l_2,\ \hat{h}'(0)=-\mu w_{y}(0, l_2)=h_1,
\]
and thus $\hat{h}'\in C^{\frac{\alpha}{2}}([0,T])$, which satisfies
\begin{equation}\label{eq2}
\|\hat{h}'\| _{C^{\frac{\alpha}{2}}([0,T])}\leqslant \mu C_{1}=:C_{2}.
\end{equation}
Similarly $\hat{g}'\in C^{\frac{\alpha}{2}}([0,T])$,
which satisfies
\begin{equation}\label{eq3}
\|\hat{g}'\| _{C^{\frac{\alpha}{2}}([0,T])}\leqslant \mu C_{1}=:C_{2}.
\end{equation}

\smallskip

$ Step\ 2$.  For any given triple $(g,h)\in \mathcal{D}$, we define an operator $ \mathcal{F}$ by
\[
 \mathcal{F}(g,h)=(\hat{g},\hat{h}).
\]
Clearly, $\mathcal{F}$ is continuous in $\mathcal{D}$, and $(g,h)\in \mathcal{D}$ is a fixed point of $\mathcal{F}$ if and
only if $(w,g,h)$ solves \eqref{lin1} and \eqref{lghin1}. We will show that if $ T>0$ is small enough, then $\mathcal{F}$ has
a fixed point by using the Schauder fixed point theorem.

Firstly, it follows from \eqref{eq2} and \eqref{eq3} that
\[
\|\hat{h}'-h_{1}\|_{C([0,T])}\leqslant
C_{2}T^{\frac{\alpha}{2}},\ \|\hat{g}'+h_{2}\|_{C([0,T])}\leqslant C_{2}T^{\frac{\alpha}{2}}.
\]
Thus if we choose $T\leqslant \min\big\{\frac{h_{0}}{16(1+ h_{1} +h_{2})},\ \tau, \ C^{-\frac{2}{\alpha}}_{2}\big\}$, then $\mathcal{F}$ maps
$\mathcal{D}$ into itself. Consequently, $\mathcal{F}$ has at least one fixed point by using the Schauder fixed point theorem, which implies
that \eqref{lin1} and \eqref{lghin1} have at least one solution $(w,g,h)$ defined in $[0,T]$. Moreover, by the Schauder estimates, we have
additional regularity for $(w, g, h)$ as a solution of \eqref{lin1} and \eqref{lghin1}, namely,
\[
(w,g,h)\in C^{1+\alpha/2,2+\alpha}((0,T]\times[l_1,l_2])\times C^{1+\alpha/2}((0,T]) \times C^{1+\alpha/2}((0,T])
\]
and for any given $0<\varepsilon<T$, there holds
\[
\|w\|_{C^{1+\alpha/2,2+\alpha}([\varepsilon,T]\times[l_1,l_2])}\leqslant C_3,
\]
where $C_3$ is a constant dependent on $\varepsilon$, $ g(\theta)$, $h(\theta)$, $\alpha$ and $\| \phi\|_{C^{1,2}}$.
Thus we deduce a local classical solution $(u,g,h)$ of \eqref{p} by $(w,g,h)$, and $u\in C^{1+\alpha/2,2+\alpha}((0,T]\times[g(t),h(t)])$
satisfies
\[
\|u\|_{C^{1+\alpha/2,2+\alpha}([\varepsilon,T]\times[g(t),h(t)])}\leqslant C_3.
\]

\smallskip

$ Step\ 3$. We will prove the uniqueness of solutions of \eqref{p}.
Let $(u_i,g_i,h_i)$, $i=1,2$, be two solutions of \eqref{p} and set
\[
w_i(t,y):=u_i(t,y+\xi_{1}(y)(h_i(t)-l_2)+\xi_{2}(y)(g_i(t)-l_1)).
\]
Then it follows from  \eqref{eq1}, \eqref{eq2} and \eqref{eq3} that
\[
\|w_i\| _{W^{1,2}_p(\Omega_{T})}+\|w_i\| _{C^{\frac{1+\alpha}{2},{1+\alpha}}(\Omega_{T})}\leqslant C_{1},\ \
\|h'_i\| _{C^{\frac{\alpha}{2}}([0,T])}\leqslant C_{2},
\ \ \|g'_i\| _{C^{\frac{\alpha}{2}}([0,T])}\leqslant C_{2}.
\]
Set
\[
\tilde{w}(t,y):=w_{1}(t,y)-w_{2}(t,y), \ \ \tilde{g}(t):=g_1(t)-g_2(t),\ \mbox{ and }\ \tilde{h}(t):=h_1(t)-h_2(t),
\]
then we find that $\tilde{w}(t,y)$ satisfies that
\begin{equation}\label{p1}
\left\{
\begin{array}{ll}
 \tilde{w}_{t} -A_{2}(t,y)\tilde{w}_{yy} + B_{2}(t,y)\tilde{w}_{y}=\tilde{f}(t,y), &  y\in(l_1, l_2),\ t\in(0, T),\\
 \tilde{w}(t,l_1)=\tilde{w}(t,l_2)= 0,  &  t\in(0, T),\\
\tilde{w}(\theta,y) =0 ,& y\in[l_1, l_2],\ \theta\in[-\tau, 0],
\end{array}
\right.
\end{equation}
where
\[
\tilde{f}(t,y)=(A_{1}-A_{2})(w_{1})_{yy}-(B_{1}-B_{2})(w
_{1})_{y}+f(w_1(t-\tau,y))-f(w_2(t-\tau,y))- d \tilde{w},
\]
and $A_i$ and $B_i$ are the coefficients of problem \eqref{lin1} with
$(w_i,g_i,h_i)$ instead of $(w,g,h)$.

Recalling that $T\leqslant \tau$, then $f(w_1(t-\tau,y))-f(w_2(t-\tau,y))=0$ for all $(t,y)\in \Omega_{T}$,
thus
\[
\tilde{f}(t,y)=(A_{1}-A_{2})(w_{1})_{yy}-(B_{1}-B_{2})(w_{1})_{y}- d \tilde{w}.
\]
Thanks to this, we can apply the $L^p$ estimates for parabolic equations to deduce that
\begin{equation}\label{sobem}
\|\tilde{w}\|_{W^{1,2}_p(\Omega_{T})}\leqslant C_4 (\| \tilde{g}\|_{C^1([0,T])}+\| \tilde{h}\|_{C^1([0,T])})
\end{equation}
with $C_4$ depending on $C_1$ and $C_2$. By a similar argument as in \cite{W}, we obtain that
\[
\| \tilde{w}\|_{C^{\frac{1+\alpha}{2},{1+\alpha}}(\Omega_{T})}\leqslant C \|\tilde{w}\|_{W^{1,2}_p(\Omega_{T})}
\]
for some positive constant $C$ independent of  $T^{-1}$. Thus
\begin{equation}\label{sobem}
\| \tilde{w}\|_{C^{\frac{1+\alpha}{2},{1+\alpha}}(\Omega_{T})}\leqslant C C_4 (\| \tilde{g}\|_{C^1([0,T])}+\| \tilde{h}\|_{C^1([0,T])})
\end{equation}
Since $\tilde{h}'(0)=h'_{1}(0)-h'_{2}(0)=0$, then
\[
\| \tilde{h}'\|_{C^{\frac{\alpha}{2}}([0,T])}=\mu \| \tilde{w}_{y}\|_{C^{\frac{\alpha}{2},0}(\Omega_{T})}\leqslant
\mu \| \tilde{w}\|_{C^{\frac{1+\alpha}{2},{1+\alpha}}(\Omega_{T})}.
\]
This, together with \eqref{sobem}, implies that
\[
\| \tilde{h}\|_{C^1([0,T])}\leqslant 2T^{\frac{\alpha}{2}} \| \tilde{h}'\|_{C^{\frac{\alpha}{2}}([0,T])}\leqslant C_5T^{\frac{\alpha}{2}} (\| \tilde{g}\|_{C^1([0,T])}+\| \tilde{h}\|_{C^1([0,T])}),
\]
where $C_5=2\mu C C_4$. Similarly, we have
 \[
\| \tilde{g}\|_{C^1([0,T])}\leqslant  C_5T^{\frac{\alpha}{2}} (\| \tilde{g}\|_{C^1([0,T])}+\| \tilde{h}\|_{C^1([0,T])}),
\]
As a consequence, we deduce that
\[
\| \tilde{g}\|_{C^1([0,T])}\|+\| \tilde{h}\|_{C^1([0,T])}
\leqslant 2C_5 T^{\frac{\alpha}{2}} (\| \tilde{g}\|_{C^1([0,T])}+\| \tilde{h}\|_{C^1([0,T])}).
\]
Hence for
\[
T:=\min\Big\{\frac{h_{0}}{ 16(1+h_{1}+h_{2})},\ \tau,\ C^{-\frac{2}{\alpha}}_{2},\ (4C_5)^{-\frac{2}{\alpha}}\Big\},
\]
we have
\[
\| \tilde{g}\|_{C^1([0,T])}\|+\| \tilde{h}\|_{C^1([0,T])}
\leqslant \frac{1}{2} (\| \tilde{g}\|_{C^1([0,T])}+\| \tilde{h}\|_{C^1([0,T])}).
\]
This shows that $\tilde{g}\equiv 0 \equiv \tilde{h}$ for $0\leqslant t\leqslant T$, thus $\tilde{w}\equiv0$ in
$[0,T]\times[l_1,l_2]$. Consequently, the uniqueness of solution of \eqref{p} is established, which ends the proof of this theorem.
\end{proof}

\begin{lem}\label{lem:global}
Assume that {\bf(H)} holds. Then every positive solution $(u, g, h)$ of problem \eqref{p} exists and is unique for all
$t\in(0, \infty)$.
\end{lem}

\begin{proof}
Let $[0, T_{max})$ be the maximal time interval in which the solution exists. In view of Theorem \ref{thm:local}, it remains to
show that $T_{max}=\infty$. We proceed by a contradiction argument and assume that $T_{max}<\infty$.  Thanks to the choice of
the initial data, the comparison principle implies that $u(t,x)\leqslant u^*$ for $(t,x)\in(0,T_{max})\times[g(t),h(t)]$.
Construct the auxiliary function
\[
\bar{u}(t,x)=u^*\big[2M(h(t)-x)-M^{2}(h(t)-x)^{2}\big],\ \ \ t\in[-\tau,T_{max}),\ x\in[h(t)-M^{-1},h(t)]
\]
where
\[
M:=\max\Big\{\sqrt{d},\ \frac{2}{h(-\tau)-g(-\tau)},\ \frac{4}{3u^*}\max_{-\tau\leqslant \theta\leqslant 0}\|\phi(\theta,\cdot)\|_{C^1([g(\theta),h(\theta)])}\Big\}.
\]
It follows the proof of \cite[Lemma 2.2]{DuLin} to prove that there is a constant $C_0$ independent on $T_{max}$
such that $h'(t)\leqslant C_0$ for $t\in (0, T_{max})$. The proof for $-g'(t)\leqslant C_0$ for $t\in (0, T_{max})$ is parallel.

Let us  now fix $\epsilon\in(0,T_{max})$. Similar to the proof of Theorem \ref{thm:local}, by standard $L^p$ estimate, the Sobolev
embedding theorem and the H\"{o}lder estimates for parabolic equation, we can find $C_1>0$ depending only on $\epsilon$, $T_{max}$,
$u^*$, $ h_0$,  $\| \phi\|_{C^{1,2}([-\tau,0]\times[g(\theta),h(\theta)])}$ and $C_0$ such that
\[
||u||_{C^{1+\alpha/2,2+\alpha}([\varepsilon,T_{max}]\times[g(t), h(t)])}\leqslant C_1.
\]
This implies that $(u,g,h)$ exists on $[0,T_{max}]$. Choosing $t_n\in(0,T_{max})$ with $t_n\nearrow T_{max}$, and regarding
$(u(t_n-\theta, x), h), g(t_n-\theta), h(t_n-\theta))$ for $\theta\in[0,\tau]$ as the initial function, it then follows from
the proof of Theorem \ref{thm:local} that there exists $s_0>0$ depending on $C_0$, $C_1$ and $u^*$ independent of $n$ such that
problem \eqref{p} has a unique solution $(u, g, h)$ in $[t_n, t_n+s_0]$. This yields that the solution $(u,g,h)$ of \eqref{p}
can be extended uniquely to $[0,t_n+s_0)$. Hence $t_n+s_0>T_{max}$ when $n$ is large. But this contradicts the assumption,
which ends the proof of this lemma.
\end{proof}

\smallskip

\noindent
 {\bf Proof of Theorem \ref{wellposedness}:} Combining Theorem \ref{thm:local} and Lemma  \ref{lem:global}, we complete the proof.{\hfill $\Box$}
\subsection{Comparison Principle}\label{subsec:cp}
In this subsection, we establish the comparison principle, which will be used in the rest of this paper. Let us
 start with the following result.
\begin{lem}
\label{lem:comp1} Suppose that {\bf{(H)}} holds, $T\in (0,\infty)$,
$\overline g,\ \overline h\in C^1([-\tau,T])$, $\overline u\in C(\overline D_T)
\cap C^{1,2}(D_T)$ satisfies $\overline u \leqslant u^*$ in $\overline D_T$
with $D_T=\{(t,x)\in\R^2: -\tau<t\leqslant T,\ \overline g(t)<x<\overline h(t)\}$, and
\begin{eqnarray*}
\left\{
\begin{array}{lll}
\overline u_{t} \geqslant  \overline u_{xx} -d \overline u+f(\overline u(t-\tau, x)),\; & 0<t \leqslant T,\
\overline g(t)<x<\overline h(t), \\
\overline u= 0,\quad \overline g'(t)\leqslant -\mu \overline u_x,\quad &
0<t \leqslant T, \ x=\overline g(t),\\
\overline u= 0,\quad \overline h'(t)\geqslant -\mu \overline u_x,\quad
&0<t \leqslant T, \ x=\overline h(t).
\end{array} \right.
\end{eqnarray*}
If $[g(\theta), h(\theta)]\subseteq [\overline g(\theta), \overline h(\theta)]$ for $\theta\in[-\tau,0]$ and $\overline u(\theta,x)\in C^{1,2}([-\tau,0]\times[\overline g(\theta),
\overline h(\theta)])$ satisfies
\[
\phi(\theta,x)\leqslant \overline u(\theta,x) \leqslant u^*\ \ \mbox{ in } [-\tau,0]\times[g(\theta),h(\theta)],
\]
then the solution $(u,g, h)$ of problem \eqref{p} satisfies $g(t)\geqslant \overline g(t)$, $h(t)\leqslant \overline h(t)$ in $(0,T]$, and
\begin{align*}
u(t,x)\leqslant  \overline u(t,x)\ \ \mbox{ for } (t,x)\in (0, T]\times(g(t), h(t)).
\end{align*}
\end{lem}
\begin{proof}
We integrate the ideas of \cite[Lemma 5.7]{DuLin} and \cite[Corollary 5]{MS} to deal with free boundary and time delay.

Firstly, for small $\epsilon>0$, let $(u_\epsilon,g_\epsilon,h_\epsilon)$ denote the unique solution of \eqref{p} with $g(\theta)$ and $h(\theta)$ replaced by
$g_\epsilon(\theta):=g(\theta)(1-\epsilon)$ and $h_\epsilon(\theta):=h(\theta)(1-\epsilon)$ for $\theta\in[-\tau,0]$, respectively, with $\mu$ replaced by
$\mu_\epsilon:=\mu(1-\epsilon)$, and with $\phi(\theta,x)$ replaced by some $\phi_\epsilon(\theta,x)\in C^{1,2}([-\tau,0]\times[g_\epsilon(\theta),h_\epsilon(\theta)])$, satisfying
\[
0<\phi_\epsilon(\theta,x)\leqslant \phi(\theta,x),\ \ \ \phi_\epsilon(\theta,g_\epsilon(\theta))=\phi_\epsilon(\theta,h_\epsilon(\theta))=0 \ \ \mbox{ for }\
\theta\in[-\tau,0],\ x\in[g_\epsilon(\theta),h_\epsilon(\theta)],
\]
and for any fixed $\theta\in[-\tau,0]$ as $\epsilon\to 0$, $\phi_\epsilon (\theta,x )\to \phi(\theta,x)$
in the $C^2([g(\theta),h(\theta)])$ norm.

We claim that $h_\epsilon(t)<\overline h(t)$, $g_\epsilon(t)> \overline g(t)$ and $u_\epsilon(t,x)<\overline u(t,x)$ for all $t\in[0,T]$
and $x\in[g_\epsilon(t),h_\epsilon(t)]$. Obviously, this is true for all small $t>0$. Now, let us use an indirect argument and suppose that
the claim does not hold, then there exists a first $t^*\in(0,T]$ such that
\begin{align*}
 u_\epsilon(t,x)< \overline{u}(t,x)\ \  \mbox{ for }\  t\in [0, t^*),\ x\in [g_\epsilon(t), h_\epsilon(t)]\subset (\overline {g}(t), \overline {h}(t)),
\end{align*}
and there is some $x^*\in[g_\epsilon(t^*),h_\epsilon(t^*)]$ such that $u_\epsilon(t^*,x^*)=\overline{u}(t^*,x^*)$.

Later, let us compare $u_\epsilon$ and $\overline u$ over the region
\[
\Omega_{t^*}:=\{(t,x)\in\R^2: 0<t\leqslant t^*,\ g_\epsilon(t)< x < h_\epsilon(t)\}.
\]
An direct computation shows that for $(t,x)\in \Omega_{t^*}$,
\[
(\overline{u}-u_\epsilon)_t- (\overline{u}-u_\epsilon)_{xx}+d (\overline{u}-u_\epsilon)\geqslant
f(\overline{u}(t-\tau,x))-f(u_\epsilon(t-\tau,x))\geqslant 0,
\]
it then follows from the strong maximum principle that
\begin{equation}\label{mao1}
u_\epsilon(t,x)<\overline{u}(t,x)\ \ \mbox{ in } \Omega_{t^*}.
\end{equation}
Thus either $x^*=h_\epsilon(t^*)$ or $x^*=g_\epsilon(t^*)$. Without loss of generality we may assume that
$x^*=h_\epsilon(t^*)$, then $\overline{u}(t^*,h_\epsilon(t^*))=u_\epsilon(t^*,h_\epsilon(t^*))=0$.
This, together with \eqref{mao1}, implies that $\overline{u}_x(t^*,h_\epsilon(t^*))\leqslant (u_\epsilon)_x(t^*,h_\epsilon(t^*))$,
from which we obtain that
\begin{equation}\label{uhq}
h'_\epsilon(t^*)=-\mu_\epsilon(u_\epsilon)_x(t^*,h_\epsilon(t^*))<-\mu \overline{u}_x(t^*,h_\epsilon(t^*))= \overline{h}'(t^*).
\end{equation}
As $h_\epsilon(t)< \overline{h}(t)$ for $t\in[0,t^*)$ and $h_\epsilon(t^*)=\overline{h}(t^*)$, then $h'_\epsilon(t^*)\geqslant \overline{h}'(t^*)$,
which contradicts \eqref{uhq}. This proves our claim.

Finally, thanks to the unique solution of \eqref{p} depending continuously on the parameters in \eqref{p}, as $\epsilon \to 0$,
$(u_\epsilon,g_\epsilon,h_\epsilon)$ converges to $(u,g,h)$, the unique of solution of \eqref{p}. The desired result then follows
by letting $\epsilon\to 0$ in the inequalities $u_\epsilon< \overline{u},\ g_\epsilon> \overline{g}$ and $h_\epsilon< \overline{h}$.
\end{proof}

By slightly modifying the proof of Lemma \ref{lem:comp1}, we obtain a variant of Lemma \ref{lem:comp1}.
\begin{lem}
\label{lem:comp2} Suppose that {\bf {(H)}} holds, $T\in
(0,\infty)$, $\overline g,\, \overline h\in C^1([-\tau,T])$, $\overline
u\in C(\overline D_T)\cap C^{1,2}(D_T)$ satisfies $\overline u \leqslant u^*$ in $\overline D_T$
with $D_T=\{(t,x)\in\R^2: -\tau<t\leqslant T,\ \overline g(t)<x<\overline h(t)\}$, and
\begin{eqnarray*}
\left\{
\begin{array}{lll}
\overline u_{t} \geqslant  \overline u_{xx} -d \overline u+f(\overline u(t-\tau, x)),\; &0<t \leqslant T,\
\overline g(t)<x<\overline h(t), \\
\overline u\geqslant u, &0<t \leqslant T, \ x= \overline g(t),\\
\overline u= 0,\quad \overline h'(t)\geqslant -\mu \overline u_x,\quad
&0<t \leqslant T, \ x=\overline h(t),
\end{array} \right.
\end{eqnarray*}
with $\overline g(t)\geqslant g(t)$ in $[0,T]$, $h(\theta)\leqslant \overline
h(\theta)$, $\phi(\theta,x)\leqslant \overline u(\theta,x)$ for $\theta\in[-\tau,0]$ and
$x\in[\overline g(\theta),h(\theta)]$, where  $(u,g, h)$ is a solution to \eqref{p}. Then
\[
\mbox{ $h(t)\leqslant \overline h(t)$ in $(0, T]$,\quad $u(x,t)\leqslant
\overline u(x,t)$ for $(t,x)\in (0, T]\times(g(t),h(t))$.}
\]
\end{lem}

\begin{remark}
\label{rem5.8}\rm The function $\overline u$, or the triple
$(\overline u,\overline g,\overline h)$, in Lemmas \ref{lem:comp1}
and \ref{lem:comp2} is often called a supersolution to \eqref{p}.
A subsolution can be defined analogously by reversing all the
inequalities. There is a symmetric version of Lemma~\ref{lem:comp2},
where the conditions on the left and right boundaries are
interchanged. We also have corresponding comparison results for
lower solutions in each case.
\end{remark}

\section{Semi-waves}\label{subsec:semiwave}

This section is devoted to proving the existence and uniqueness of a semi-wave $q(z)$ of \eqref{sw11}, which will be used
to construct some suitable sub- and supersolutions to study the asymptotic profiles of spreading solutions of \eqref{p}.
Let us consider the following nonlocal elliptic problem
\begin{equation}\label{semiwave}
\left\{
 \begin{array}{ll}
  q'' - cq'-d q+ f( q(z-c\tau))=0, & z>0,\\
 q(z)=0,  & z\leqslant 0,\\
 \end{array}
 \right.
\end{equation}
where $c\geqslant 0$ is a constant.

If  $z$ is understood as the time variable, then we may regard problem \eqref{semiwave} as a time-delayed dynamical system in the phase space $C([-c\tau,0],\R^2)$. When $c\tau=0$, the phase space reduces to $\R^2$ and it follows from the phase plane analysis that \eqref{semiwave} admits a unique positive solution $q_0(z)$, which is increasing in $z$ and $q_0(z)\to u^*$ as $z\to \infty$. When $c\tau>0$, the phase space is of infinite dimension and the positivity and boundedness of the unique solution are not clear.

\begin{prop}\label{prop:semiwave}
Suppose {\bf{(H)}} holds. For any given constant $c> 0$, problem \eqref{semiwave} has a maximal nonnegative solution $q_c$. Moreover,
either $q_c(z)\equiv 0$ or $q_c(z)> 0$ in $(0,\infty)$. Furthermore, if $q_c>0$, then it is the unique positive solution of \eqref{semiwave},
$q_c'(z)>0$ in $(0,\infty)$ and $q_c(z)\to u^*$ as $z\to\infty$, in addition, for any given constant $c_1<c$, one has $q_c(z)<q_{c_1}(z)$
for $z\in(0,\infty)$, and  $q'_c(0)<q'_{c_1}(0)$.
\end{prop}
\begin{proof}
We divide the proof into four steps.

\smallskip

$ Step \ 1$. Problem \eqref{semiwave} always has a maximal nonnegative solution $\overline{q}$ and it satisfies
\[
\overline{q}\leqslant u^*\ \ \mbox{ for } z\in[0,\infty).
\]
Clearly, $0$ is a nonnegative solution of \eqref{semiwave}. For any $l>0$, consider the following problem:
\begin{equation}\label{semiwavel}
\left\{
 \begin{array}{ll}
   w'' - cw'-d w+ f( w(z-c\tau))=0, & 0<z<l,\\
w(l)=u^*,\ \ \  w(z)=0,  \ \  z\leqslant 0.
 \end{array}
 \right.
\end{equation}
It is well known problem \eqref{semiwavel} admits a unique solution $w^l(z)>0$ for $z\in(0,l]$. Applying the maximal
principle, we can deduce that $w^l(z)\leqslant u^*$ for $z\in[0,l]$. Moreover, it is easy to check that $w^l(z)$ is
decreasing in $l>0$ and increasing in $z\in[0,l]$ and
\[
w^l(z)\to W(z)\ \ \mbox{ as } l\to\infty,
\]
where $W(z)$ is a nonnegative solution of problem \eqref{semiwave} and it satisfies $W(z)\leqslant u^*$ for $z\in[0,\infty)$.

In what follows, we want to prove that $W$ is the maximal nonnegative solution of \eqref{semiwave}. Let $q$ be an arbitrary
nonnegative solution of \eqref{semiwave}, then $q(z)\leqslant u^*$ for $z\in[0,\infty)$. If $q\equiv 0$, then $q\leqslant W$.
Suppose now $q\geqslant, \not\equiv 0$, then $q>0$ in $(0,\infty)$. Let us show $q(z)\leqslant W(z)$ for $z\in[0,\infty)$.

Firstly, for any fixed $l>0$ we can find $M>0$ large such that $Mw^l(z)\geqslant q(z)$ for $z\in[0,l]$. We claim that the above
inequality holds for $M=1$. On the contrary, define
\[
M_0:=\inf\{M>0:\ Mw^l(z)\geqslant q(z)\ \ \mbox{ for } z\in[0,l]\},
\]
then $M_0>1$ and $M_0w^l(z)\geqslant,\not\equiv q(z)$ for $z\in[0,l]$. Thanks to the monotonicity of $w^l(z)$ in $z\in[0,l]$,
then there is $z_0\in(0,l)$ such that $M_0w^l(z_0)=u^*$ and $M_0w^l(z)<u^*$ for $z\in[0,z_0)$. It is easy to check that $q(z_0)<u^*$.
Then the strong maximal principle yields that $M_0 (w^l)'(0)>q'(0)$ and $M_0w^l(z)>q(z)$ for $z\in(0,z_0]$. Thus we can find a constant
$0<\epsilon\ll1$ such that
\begin{equation}\label{Mqw}
M_1:=M_0(1+\epsilon)^{-1}>1, \ \ M_1w^l(z)>q(z)\ \ \mbox{ for } z\in(0,z_0],
\end{equation}
and$ M_1w^l(z_0+\tilde{z})> u^*$ for $\tilde{z}=\min\{c\tau,\ l-z_0\}$.
So there is $z_1\in(0,\tilde{z}]$ such that $M_1w^l(z_0+z_1)= u^*$ and $M_1w^l(z_0+z)> u^*$ for $z\in (z_1,l-z_0]$.

Later, we want to prove that $M_1 w^l(z)>q(z)$ for all $z\in(z_0,l]$. Combining the definition
of $z_1$, we only need to prove $M_1 w^l(z)\geqslant q(z)$ for all $z\in(z_0,z_0+z_1]$. Since $M_1 w^l(z)\geqslant
q(z)$ for $z=z_0+z_1$ and  $z=z_0$, and for $z\in(z_0,z_0+z_1)$,
\begin{eqnarray*}
&& \big(M_1w^l-q\big)''-c\big(M_1w^l-q\big)'-d \big(M_1 w^l-q\big)\\
&=& f(q(z-c\tau))-M_1f\big( w^l(z-c\tau)\big)\\
& \leqslant& f(q(z-c\tau))-f\big(M_1 w^l(z-c\tau)\big) \leqslant 0,
\end{eqnarray*}
where the monotonicity of $f(v)$ in $v\in[0,u^*]$ and the fact where $M_1w^l(z-c\tau)\geqslant q(z-c\tau)$ for $z\leqslant z_0+z_1$
are used. The comparison principle yields that $M_1 w^l(z)\geqslant q(z)$ for all $z\in[z_0,z_0+z_1]$. This, together with
the definition of $z_1$ and \eqref{Mqw}, yields that $M_1 w^l(z)\geqslant q(z)$ for all $z\in(0,l]$, which contradicts the definition
of $M_0$. Thus we have proved that $w^l(z)\geqslant q(z)$ for $z\in[0,l]$.

Finally, letting $l\to\infty$, we deduce that
\[
W(z)\geqslant q(z)\ \ \mbox{ for } z\in[0,\infty),
\]
as we wanted. Thus Step 1 is proved.

\smallskip

$ Step \ 2$. For any $c\geqslant 0$, if $q$ is a positive solution of \eqref{semiwave}, then $q_+'(0)>0$, $q'(z)>0$ for $z\in(0,\infty)$,
and $q(z)\to u^*$ as $z\to\infty$.

Since $q>0$ for $z>0$, then the Hopf lemma can be used to deduce $q_+'(0)>0$, it follows that $q'(z)>0$ for all small $z>0$.
Setting
\[
\gamma^*:=\sup\{\gamma>0:\ q(2\gamma-z)>q(z)\ \mbox{ for } z\in[0,\gamma),\ \ q'(z)>0\ \mbox{ for } z\in(0,\gamma]\}.
\]
In the following, we shall show $\gamma^*=\infty$. Suppose by way of contradiction that $\gamma^*\in(0,\infty)$, then
\[
q(2\gamma^*-z)\geqslant q(z), \ \mbox{ and }\ q'(z)\geqslant 0\ \ \mbox{ for } z\in[0,\gamma^*].
\]
Define $\tilde{q}(z)=q(2\gamma^*-z)$ for $z\in[\gamma^*,2\gamma^*]$, then
\[
 \tilde{q}''-c\tilde{q}'-d\tilde{q}+f(\tilde{q}(z-c\tau))=-2cq_\xi,\ \ \ \xi=2\gamma^*-z\in[0,\gamma^*].
\]
Let us set
\[
Q(z;\gamma^*)=Q(z)=\tilde{q}(z)-q(z)=q(\xi)-q(2\gamma^*-\xi).
\]
Then $Q\leqslant 0$ for $z\in[\gamma^*,2\gamma^*]$ and it satisfies
\begin{equation}\label{Qqq}
\left\{
 \begin{array}{ll}
 Q'' - cQ'-d Q=f(q(z-c\tau))-f(\tilde{q}(z-c\tau))-2cq_\xi\leqslant 0, & \gamma^*\leqslant z\leqslant 2\gamma^*,\\
Q(\gamma^*)=0,\ \ \  Q(2\gamma^*)=-q(2\gamma^*)<0.
 \end{array}
 \right.
\end{equation}
The strong maximal principle and the Hopf lemma imply that
\[
Q(z)<0,\ \ \ z\in(\gamma^*,2\gamma^*],\ \ \ Q'(\gamma^*)<0.
\]
It follows the continuity that for all small $\varepsilon\geqslant 0$,
\[
Q'(\gamma^*+\varepsilon;\gamma^*+\varepsilon)<0,\ \ \
Q(z;\gamma^*+\varepsilon)<0\ \ \mbox{ for } z\in(\gamma^*+\varepsilon,2\gamma^*+2\varepsilon],
\]
which implies that $q(2\gamma^*+2\varepsilon-\xi)>q(\xi)$ for $\xi\in[0,\gamma^*+\varepsilon)$. Moreover, since
$Q'(\gamma^*+\varepsilon;\gamma^*+\varepsilon)=-2q'(\gamma^*+\varepsilon)$, it then follows that $q'(\gamma^*+\varepsilon)>0$.
But these facts contradict the definition of $\gamma^*$. Thus the monotonicity of positive solutions of \eqref{semiwave} is
established.

Next, we consider the asymptotic behavior of positive solution $q$ of \eqref{semiwave}. The monotonicity of $q$ implies that
there is a constant $a>0$ such that $\lim_{z\to\infty} q(z)=a$. We claim that $a=u^*$. For any sequence
$\{z_n\}$ with $z_n\to\infty$ as $n\to\infty$, define $q_n(z)=q(z+z_n)$. Then $q_n$ solves the same equation as $q$ but over
$(-z_n,\infty)$. Since $q_n\leqslant u^*$, it follows that there is a subsequence of $\{q_n\}$ (still denoted by $\{q_n\}$)
such that $q_n\to \hat{q}$ locally in $C^2(\R)$  as $n\to\infty$, and $\hat{q}$ is a solution of
\[
 v''-cv'-d v+f(v(z-c\tau))=0,\ \ \ z\in\R.
\]

On the other hand, it follows from $\lim_{z\to\infty}q(z)=a$ that $ \hat{q}\equiv a$, which implies that $a=u^*$, as we wanted.
Thus this completes the proof of Step 2.

\smallskip

$ Step \ 3$. We show that problem \eqref{semiwave} has at most one positive solution.

Suppose problem \eqref{semiwave} has two positive solutions $q_1$ and $q_2$, then $0<q_i<u^*$ in $(0,\infty)$, and
$q_i(z)\to u^*$ as $z\to\infty$ for $i=1,\ 2$. Define
\[
\rho^*:=\inf\left\{\frac{q_1(z)}{q_2(z)}:z>0\right\}.
\]
From Step 2 we have $(q_i)_+'(0)>0$, $i=1, 2$. Then by L'H\^{o}pital's rule we obtain
$\lim_{z\downarrow 0}\frac{q_1(z)}{q_2(z)}>0$, which together with $\lim_{z\to+\infty}\frac{q_1(z)}{q_2(z)}=1$
implies that $\rho^*\in (0,1]$. Next we show $\rho^*=1$. Indeed, assume for the sake of contraction that $\rho^*\in (0,1)$.
Define
\[
w(z):=q_1(z)-\rho^*q_2(z).
\]
Then $w(z)\geqslant 0$ for $z\geqslant 0$,  $w(0)=0$, $w(+\infty)=(1-\rho^*)u^*>0$ and
\[
w''-cw'-dw=-f(q_1(z-c\tau))+\rho^*f(q_2(z-c\tau))\leqslant 0,
\]
where the sub-linearity and monotonicity of $f(z)$ for $z\in(0,u^*)$ are used. By Hopf's lemma,
we see that $0<w'(0)=(q_1)_+'(0)-\rho^* (q_2)_+'(0)$, which implies that $\lim_{z\downarrow 0}\frac{q_1(z)}{q_2(z)}>\rho^*$.
Thus, in view of the definition of $\rho^*$, we have an $z_0\in (0,+\infty)$ such that $w(z_0)=0$.
By the elliptic strong maximum principle, we infer that $w(z)\equiv 0$ for $z>0$, a contradiction
to $w(+\infty)>0$. Therefore, $\rho^*=1$, and hence, $q_1(z)\geqslant q_2(z)$. Changing the role of
$q_1$ and $q_2$ and repeating the above arguments, we obtain $q_2(z)\geqslant q_1(z)$. The uniqueness is proved.

\smallskip

$ Step \ 4$. Let us consider the monotonicity of positive solutions in $c$.

Assume that $q_c$ is a positive solution of \eqref{semiwave}. Choose $c_1<c$ and let $q_{c_1}$ be the maximal nonnegative solution
of \eqref{semiwave} with $c=c_1$. Since $u^*$ is a supersolution of \eqref{semiwave}, and by Step 2 we know that $q_c$ is a subsolution
of \eqref{semiwave} with $c=c_1$, in view of the uniqueness of positive solution of this problem, then we see that
$q_{c_1}(z)\geqslant q_c(z)$ for $z\in[0,\infty)$. It thus follows from the maximum principle and the Hopf lemma that
\[
q_{c_1}(z)>q_c(z)\ \ \mbox{ for } z\in(0,\infty), \ \ \mbox{ and }\ \ q'_{c_1}(0)>q'_c(0).
\]

The proof of this proposition is complete now.
\end{proof}

Next we give a necessary and sufficient condition for the existence of a positive solution of \eqref{semiwave}.  For this purpose, we need  the following property on the distribution of complex solutions to a transcendental equation.
\begin{lem}\label{lem:eigen}
Let $c>0$ and $\tau>0$. Define
\begin{equation}
\Delta_c(\lambda,\tau)=\lambda^2-c\lambda-d+f'(0)e^{-\lambda c\tau}.
\end{equation}
Then there exists $c_0(\tau)\in (0,2\sqrt{f'(0)-d})$ such that the following statements hold:
\begin{enumerate}
\item[(i)] $\Delta_c(\lambda,\tau)=0$ has a positive solution if and only if $c\geqslant c_0(\tau)$;
\item[(ii)] $\Delta_c(\lambda,\tau)=0$ has a complex solution in the domain
\begin{equation}\label{def-Omega}
\Omega:=\left\{\lambda\in \mathbb{C}: Re \lambda>0, Im \lambda \in \left(0,\frac{\pi}{c\tau}\right) \right\}
\end{equation}
provided that $c\in (0,c_0(\tau))$.
\end{enumerate}
\end{lem}
Before the proof, we note that if $\tau=0$ then $\Delta_c(\lambda,\tau)=0$ reduces to a polynomial equation of order $2$. It admits at least one positive solution if and only if $c\geqslant 2\sqrt{f'(0)-d}$ and exactly a pair of complex eigenvalues in $\Omega$ when $c\in (0,2\sqrt{f'(0)-d})$.

\begin{proof}
(i) Note that  $\Delta_c(\lambda,\tau)$ is convex in $\lambda$, decreasing in $c>0$ when $\lambda>0$,  $\Delta_0(\lambda,\tau)>0$ and $\Delta_c(\lambda,\tau)=0$ is negative for some $\lambda>0$ when $c$ is sufficiently large. Therefore, such $c_0(\tau)$ exists.

(ii) We employ a continuation method with $\tau$ being the parameter. From the proof of \cite[Theorem 2.1]{RuanWei2003}, we can infer that the solutions of $\Delta_c(\lambda,\tau)=0$ is continuous in $\tau>0$. We write $\lambda=\alpha(\tau)+i\beta(\tau)$, where $\alpha(\tau)$ and $\beta(\tau)$ are continuous in $\tau>0$. Separating the real and imaginary parts of  $\Delta_c(\lambda,\tau)=0$ yields
\begin{equation}\label{s1}
\begin{cases}
F_1(\alpha,\beta,\tau):=\alpha^2-\beta^2-c\alpha-d+f'(0)e^{-c\tau \alpha}\cos c\tau\beta=0\\
F_2(\alpha,\beta,\tau):=2\alpha\beta-c\beta-f'(0)e^{-c\tau \alpha}\sin c\tau\beta=0.
\end{cases}
\end{equation}

We proceed with four steps.

\smallskip

$ Step \ 1$. If $\tau$ is small enough, then there is a solution in $\Omega$. Indeed,
At $\tau=0$, \eqref{s1} admits a solution $(\alpha,\beta)=\left(\frac{c}{2}, \frac{ \sqrt{|c^2-(f'(0)-d)^2}|}{2}\right)$. Note that
\begin{equation}
\det
\left(
\begin{matrix}
\partial_\alpha F_1 & \partial_\beta F_1\\
\partial_\alpha F_2 & \partial_\beta F_2
\end{matrix}\right)|_{\tau=0 }
=\det
\left(
\begin{matrix}
2\alpha-c &-2\beta\\
2\beta & 2\alpha+c
\end{matrix}\right)
>0.
\end{equation}
It then follows from the implicit function theorem that for small $\tau$, $\Delta_c(\lambda,\tau)$ admits a complex solution near $\frac{c}{2}+i\frac{ \sqrt{|c^2-(f'(0)-d)^2}|}{2}$, and hence, in the open domain $\Omega$.

\smallskip

$ Step \ 2$. For any $\tau>0$, $\Delta_c(\lambda,\tau)$ admits no solution with $\beta=0$ or $\beta=\frac{\pi}{c\tau}$ when $c\tau>0$. It follows from statement (i) that there is no solution with $\beta=0$ when $c<c_0(\tau)$. If $\beta$ equals $\frac{\pi}{c\tau}$, then from the second equation of \eqref{s1} we can infer that $\alpha=\frac{c}{2}$. Substituting $\alpha=\frac{c}{2}$ and $\beta=\frac{\pi}{c\tau}$ into the first equation of \eqref{s1}, we obtain $0=-\frac{1}{4}c^2-\left(\frac{\pi}{c\tau}\right)^2-d-f'(0)e^{-c^2\tau/2}$, a contradiction.

\smallskip

$ Step \ 3$. If a solution $\alpha(\tau)+i\beta(\tau)$ touches pure imaginary axis at some $\tau=\tau^*>0$, then $\alpha'(\tau^*)>0$. We use the implicit function theorem. By direct computations, we have
\begin{eqnarray*}
&&\det
\left(
\begin{matrix}
\partial_\alpha F_1 & \partial_\beta F_1\\
\partial_\alpha F_2 & \partial_\beta F_2
\end{matrix}\right)|_{\tau=\tau^*}\\
=&&\det
\left(
\begin{matrix}
-c-c\tau f'(0)\cos c\tau\beta &-2\beta-c\tau f'(0)\sin c\tau\beta\\
2\beta+c\tau f'(0)\sin c\tau\beta& -c-c\tau f'(0)\cos c\tau\beta
\end{matrix}\right)\\
=&& [-c-c\tau f'(0)\cos c\tau\beta]^2+ [2\beta+c\tau f'(0)\sin c\tau\beta]^2 \geqslant 0,
\end{eqnarray*}
where the equality holds if and only if $-c-c\tau f'(0)\cos c\tau\beta=0$ and $2\beta+c\tau f'(0)\sin c\tau\beta=0$. Taking these two relations into \eqref{s1} with $\alpha=0$, we obtain
\begin{equation}
\begin{cases}
-\beta^2-d-\frac{1}{\tau}=0\\
-c\beta+\frac{2\beta}{c\tau}=0,
\end{cases}
\end{equation}
which is not solvable for $\beta$. Therefore,
\[
\det \left(
\begin{matrix}
\partial_\alpha F_1 & \partial_\beta F_1\\
\partial_\alpha F_2 & \partial_\beta F_2
\end{matrix}\right)|_{\tau=\tau^*}>0.
\]
On the other hand,
\[
\left(
\begin{matrix}
\partial_\tau F_1\\
\partial_\tau F_2
\end{matrix}\right)|_{\tau=\tau^*}
=-c\beta f'(0)\left(
\begin{matrix}
\sin c\tau\beta \\
\cos c\tau\beta
\end{matrix}\right)
\]
Consequently, by the implicit function theorem we have
\[
\left(
\begin{matrix}
\alpha'(\tau^*)\\
\beta'(\tau^*)
\end{matrix}\right)|_{\tau=\tau^*}
=-\left(
\begin{matrix}
\partial_\alpha F_1 & \partial_\beta F_1\\
\partial_\alpha F_2 & \partial_\beta F_2
\end{matrix}\right)^{-1}|_{\tau=\tau^*}\left(
\begin{matrix}
\partial_\tau F_1\\
\partial_\tau F_2
\end{matrix}\right)|_{\tau=\tau^*},
\]
from which we compute to have
\begin{equation}
\alpha'(\tau^*)=\frac{(2\beta^4+2d\beta^2+c^2)c}{\det \left(
\begin{matrix}
\partial_\alpha F_1 & \partial_\beta F_1\\
\partial_\alpha F_2 & \partial_\beta F_2
\end{matrix}\right)|_{\tau=\tau^*}}>0.
\end{equation}

\smallskip

$ Step \ 4$. Completion of the proof. In Steps 2 and 3, we have verified that the perturbed solution at Step 1 can not escape $\Omega$ continuously as $\tau$ increases from $0$ to $\infty$. Therefore, it always stays in $\Omega$.
\end{proof}

Based on the above results, we are ready to give the following necessary and sufficient condition for
\eqref{semiwave} to have a unique positive solution.

\begin{prop}\label{prop:qoan1}
Suppose {\bf{(H)}} holds.  Problem \eqref{semiwave} has a unique positive solution $q\in C^2([0,\infty))$ if and only if $c\in[0,c_0(\tau))$,
where $c_0(\tau)$ is given in Lemma \ref{lem:eigen}.
\end{prop}
\begin{proof}
Firstly, let us show that problem \eqref{semiwave} admits a unique positive solution when $c\in[0,c_0(\tau))$.  We employ the super- and subsolution method.
The case where $c\tau=0$ is trivial and the proof is omitted. Fix $c\in (0,c_0(\tau))$. By Lemma \ref{lem:eigen} we can infer that there exists
$\gamma>0$ such that
\begin{equation}
\widetilde{\Delta}_c(\lambda)=\lambda^2-c\lambda-d+(1-\gamma) f'(0)e^{-\lambda c\tau}=0
\end{equation}
has a solution $\lambda=\alpha+i\beta$ in $\Omega$.

{\bf Claim.} The function
\begin{equation}
\underline{v}(x):=
\begin{cases}
\delta e^{\alpha x} cos \beta x, & \beta x\in (\frac{3\pi}{2}, \frac{5\pi}{2}),\\
0, & \text{elsewhere},
\end{cases}
\end{equation}
is a subsolution provided that $\delta$ is small enough.

Indeed, for $\beta x\in (\frac{3\pi}{2}, \frac{5\pi}{2})$, we have
\begin{eqnarray*}
L[\underline{v}](x):=&&\underline{v}''(x)-c\underline{v}'(x)-d\underline{v}(x)+f(\underline{v}(x-c\tau))\\
=&& \underline{v}(x) \left[ \alpha^2-\beta^2-c\alpha -d- [2\alpha\beta -c\beta]\tan\beta x \right] +f(\underline{v}(x-c\tau))\\
=&&  -\underline{v}(x) \frac{1}{\cos \beta x}(1-\gamma)f'(0)e^{-c\tau\alpha}\cos(\beta(x-c\tau))+f(\underline{v}(x-c\tau))\\
=&&  -(1-\gamma) f'(0)\delta e^{\alpha(x-c\tau)}\cos\beta(x-c\tau)+f(\underline{v}(x-c\tau)).
\end{eqnarray*}
Choose $\delta>0$ sufficiently small such that
\[
f(\underline{v}(x-c\tau)) \geqslant (1-\gamma) f'(0) \underline{v}(x-c\tau),
\]
with which we obtain
\[
L[\underline{v}](x)\geqslant (1-\gamma)f'(0) [\underline{v}(x-c\tau)-\delta e^{\alpha(x-c\tau)}\cos\beta(x-c\tau)],\quad \beta x\in \left(\frac{3\pi}{2}, \frac{5\pi}{2}\right).
\]
Clearly, if $\beta (x-c\tau) \in \left(\frac{3\pi}{2}, \frac{5\pi}{2}\right)$, then $\underline{v}(x-c\tau)=\delta e^{\alpha(x-c\tau)}\cos\beta(x-c\tau)$, and hence, $L[\underline{v}](x)\geqslant 0$. If $\beta (x-c\tau) \not\in \left(\frac{3\pi}{2}, \frac{5\pi}{2}\right)$, then $\underline{v}(x-c\tau)=0$, and hence,
\[
L[\underline{v}](x)\geqslant -(1-\gamma)f'(0)\delta e^{\alpha(x-c\tau)}\cos\beta(x-c\tau)
\]
with $\beta (x-c\tau)\in \left(\frac{3\pi}{2}-\beta c\tau, \frac{5\pi}{2}-\beta c\tau\right)\setminus \left(\frac{3\pi}{2}, \frac{5\pi}{2}\right)$. Since $\beta c\tau\leqslant \pi$ (as proved in Lemma \ref{lem:eigen}), we obtain $\cos\beta(x-c\tau)\leqslant 0$ when $\beta (x-c\tau)\in \left(\frac{3\pi}{2}-\beta c\tau, \frac{5\pi}{2}-\beta c\tau\right)\setminus \left(\frac{3\pi}{2}, \frac{5\pi}{2}\right)$. To summarize, $L[\underline{v}](x)\geqslant 0$ for $\beta x\in \left(\frac{3\pi}{2}, \frac{5\pi}{2}\right)$ and $L[\underline{v}](x)= 0$ for $\beta x\not \in \left[\frac{3\pi}{2}, \frac{5\pi}{2}\right]$. The claim is proved.

Having such a subsolution, we can infer that \eqref{semiwave} admits a positive solution. The proof of uniqueness of the solution of \eqref{semiwave} follows from Proposition \ref{prop:semiwave}.

\smallskip

Next we show that \eqref{semiwave} does not admit a positive solution when $c\geqslant c_0(\tau)$. We employ a sliding argument. Assume for the sake of contradiction that there is a solution $q(z)$. Since $c\geqslant c_0(\tau)$, $\Delta_c(\lambda,\tau)=0$ admits a positive solution $\lambda_1$. Define $w(z)=le^{\lambda_1 z}-q(z), l>0$. Since $q(0)=0$ and $q(+\infty)=u^*$, we may choose $l$ such that $w(z)\geqslant 0$ for $z\geqslant 0$ and $w(z)$ vanishes at some $z\in (0,+\infty)$. Note that $f(u)\leqslant f'(0) u$. It then follows that
\begin{equation}
w''(z)-cw'(z)-dw(z)=-f'(0)w(z-c\tau)+[f(q(z-c\tau))-f'(0)q(z-c\tau)]\leqslant 0, \quad z \geqslant 0.
\end{equation}
By the elliptic strong maximum principle, we obtain $w(z)=0$ for $z\geqslant 0$, a contradiction. The nonexistence is proved.
\end{proof}

Based on the above results, we obtain the solvability of  \eqref{sw11}.

\begin{thm}\label{waves}
For any given $\tau>0$, let $c_0(\tau)$ be given in Lemma \ref{lem:eigen}. For each $\mu>0$, there exists a unique $c^*=c^*_\mu(\tau)\in (0, c_0(\tau))$
such that $(q_{c^*})'_+(0)=\frac{c^*}{\mu}$, where $q_{c^*}(z)$ is the unique positive solution of \eqref{semiwave} with $c$ replaced by $c^*$. Moreover, $c^*_\mu(\tau)$
is increasing in $\mu$ with
\[
\lim_{\mu\to\infty}c^*_\mu(\tau)=c_0(\tau).
\]
\end{thm}
\begin{proof}
From Propositions \ref{prop:semiwave} and \ref{prop:qoan1}, it is known that for each $c\in [0,c_0(\tau))$, problem \eqref{semiwave} admits
a unique solution $q_c(z)>0$ for $z>0$, and for any $0\leqslant c_1<c_2\leqslant c_0(\tau)$, $q_{c_1}(z)>q_{c_2}(z)$ in $(0,\infty)$.
Define
\begin{equation}\label{def-P}
P(0;c,\tau):=(q_c)_+'(0).
\end{equation}
Then $P(0;c,\tau)>0$ for all $c\in[0,c_0(\tau))$ and it decreases continuously  in
$c\in [0, c_0(\tau))$. Let $c_n\uparrow c_0(\tau)$. For each $c_n$ problem \eqref{semiwave} admits a unique solution $q_{c_n}(z)$. Clearly, $q_{c_n}$ converges to some $q^*$ and $(q_{c_n})'$ converges to $(q^*)'$ locally uniformly in $z\in[0,+\infty)$, and $q^*$ solves \eqref{semiwave} with $c=c_0(\tau)$. By the nonexistence established in Proposition \ref{prop:qoan1} we obtain $q^*\equiv 0$. In particular,
\begin{equation}
\lim_{c\uparrow c_0(\tau)} (q_c)_+'(0)=(q^*)_+'(0)=0.
\end{equation}
We now consider the continuous function
\[
\eta(c;\tau)=\eta_\mu(c;\tau):=P(0;c,\tau)-\frac{c}{\mu}\ \ \mbox{ for }\ c\in [0,c_0(\tau)).
\]
By the above discussion we know that $\eta(c;\tau)$ is strictly decreasing in $c\in[0, c_0(\tau))$. Moreover, $\eta(0;\tau)=P(0;0,\tau)>0$ and
$\lim_{c\uparrow c_0(\tau)}\eta(c;\tau)=-c_0(\tau)/\mu<0$. Thus there exists a unique $c^*=c^*_\mu(\tau)\in (0, c_0(\tau))$ such that $\eta(c^*;\tau)=0$,
which means that
\[
(q_{c^*})_+'(0)=\frac{c^*}{\mu}.
\]

Next, let us view $(c^*_\mu, c_\mu^*/\mu)$ as the unique intersection point of the decreasing curve $y=P(0;c,\tau)$ with the increasing
line $y=c/\mu$ in the $cy$-plane, then it is clear that $c^*_\mu(\tau)$ increases to $c_0(\tau)$ as $\mu$ increases to $\infty$.
The proof is complete.
\end{proof}
\begin{remark}\label{tau0}\rm
In \cite{DuLou}, the authors considered the case $\tau=0$. They obtained that for each $\mu>0$, there is a unique
$c^*=c^*_\mu(0)\in (0, c_0(0))$ such that $(q_{c^*})'_+(0)=\frac{c^*}{\mu}$, where $q_{c^*}(z)$ is the unique of \eqref{semiwave}
with $\tau=0$ and $c=c^*$, and $c_0(0)=2\sqrt{f'(0)-d}$. Moreover, $c^*_\mu(0)$
is increasing in $\mu$ with
\[
\lim_{\mu\to\infty}c^*_\mu(0)=c_0(0).
\]
 \end{remark}

In the rest of this part, we study the monotonicity of $c^*_\mu(\tau)$ in $\tau$. For any given $\tau\geqslant 0$,
the unique positive solution of \eqref{semiwave} with $c\in[0,c_0(\tau))$ may be denoted by $q_c(z;\tau)$. Now we
give the proof of Theorem \ref{thm:semiwave}.

\smallskip

\noindent
 {\bf Proof of Theorem \ref{thm:semiwave}:}
For $\tau\geqslant 0$ and $\mu>0$, let $c^*_\mu(\tau)$ be given in Theorem \ref{waves} and Remark \ref{tau0} for $\tau>0$ and $\tau=0$,
respectively. By Propositions \ref{prop:semiwave} and \ref{prop:qoan1}, we see that for $\tau\geqslant 0$ and $c\in (0,c_0(\tau))$,
problem \eqref{semiwave} admits a unique positive solution $q_c(z;\tau)$. Moreover, $q_c(z;\tau)$ is increasing in $z>0$ and
decreasing in $c\in (0,c_0(\tau))$. Let $P(0;c,\tau)$ be defined as in \eqref{def-P}.

{\bf Claim.}  For $0\leqslant \tau_1<\tau_2$ , $P(0;c,\tau_1)>P(0;c,\tau_2)$ when $c\in (0,c_0(\tau_2))$.

We postpone the proof of the claim and reach the conclusion in a few lines. Note that $c^*_\mu(\tau)$ is the unique positive solution of
$P(0;c,\tau)-\frac{c}{\mu}=0$. In view of $\lim_{c\uparrow c_0(\tau_2)} P(0;c,\tau_2)=0$, we have $c^*_\mu(\tau_2)\in (0,c_0(\tau_2))$.
If $c^*_\mu(\tau_1)\geqslant c_0(\tau_2)$, then we are done. Otherwise, $c^*_\mu(\tau_1)\in (0,c_0(\tau_2))$, which, together with the claim,
implies that
\[
\frac{c^*_\mu(\tau_1)}{\mu}=P(0;c^*_\mu(\tau_1),\tau_1)>P(0;c^*_\mu(\tau_1),\tau_2).
\]
This further implies that $c^*_\mu(\tau_1)>c^*_\mu(\tau_2)$, due to the monotonicity of $P(0;c,\tau_2)-\frac{c}{\mu}$ in $c\in (0,c_0(\tau_2))$.
Thus, $c^*_\mu(\tau)$ is decreasing in $\tau\geqslant 0$.

{\it Proof of the claim.} Since $c_0(\tau)$ is decreasing in $\tau\geqslant 0$, we see that $P(0;c,\tau_1)$ is well-defined when $c\in (0,c_0(\tau_2))$.
By the monotonicity of $q_c(z;\tau_2)$ in $z>0$, we have $q_c(z-c\tau_2;\tau_2)<q_c(z-c\tau_1;\tau_2)$. This, together with the monotonicity of $f(v)$ in $v$,
implies that $f(q_c(z-c\tau_2;\tau_2))< f(q_c(z-c\tau_1;\tau_2))$. Consequently,
\[
q_c''(z;\tau_2)-cq_c'(z;\tau_2)-dq_c(z;\tau_2)+f(q_c(z-c\tau_1;\tau_2))> 0,\quad z>0.
\]
Consider the initial value problem
\begin{equation}
\begin{cases}
v_t=v_{zz}-cv_z-dv+f(v(t,z-c\tau_1)),& t>0,\ z>0\\
v(t,z)=0, &t>0,\ z\leqslant 0\\
v(0,z)=q_c(z;\tau_2)
\end{cases}
\end{equation}
By the maximum principle we know that $v(t,z)$ is nondecreasing in $t\geqslant 0$ and its limit $v^*(z)$ as $t\to\infty$ satisfies \eqref{semiwave} with
$\tau=\tau_1$. By the uniqueness established in Proposition \ref{prop:semiwave}, we obtain $v^*(z)=q_c(z;\tau_1)$. Therefore,
\begin{equation}
q_c(z;\tau_2)=v(0,z)\leqslant v(t,z)\leqslant v(+\infty,z)=v^*(z)=q_c(z;\tau_1).
\end{equation}
The claim is proved. {\hfill $\Box$}

\section{Long time behavior of the solutions}\label{seclo}
In this section we study the asymptotic behavior of solutions of \eqref{p}. Firstly,  we give some sufficient
conditions for vanishing and spreading. Next, based on these results, we prove the spreading-vanishing dichotomy result
of \eqref{p}. Let us start this section with the following equivalent conditions for vanishing.

\begin{lem}\label{lemvansmall}
Assume that {\bf(H)} holds. Let $(u,g,h)$ be a solution of \eqref{p}. Then the following three assertions are equivalent:
$$
{\rm (i)}\  h_\infty \mbox{ or } g_\infty \mbox{ is finite};\qquad
{\rm (ii)}\  h_\infty-g_\infty\leqslant  \pi/\sqrt{f'(0)- d }; \qquad  {\rm (iii)}\ \lim_{t\to\infty}\|u(t,\cdot)\|_{L^\infty ([g(t),h(t)])}= 0.
$$
\end{lem}

\begin{proof}
``(i)$\Rightarrow$ (ii)". Without loss of generality we assume $h_\infty < -\infty$ and prove (ii) by contradiction.
Assume that $h_\infty-g_\infty > \pi/\sqrt{f'(0)- d }$, then there exists $t_1 \gg 1$ such that
\[
h(t_1) - g(t_1) > \frac{\pi}{\sqrt{f'(0)- d }}.
\]

Let us consider the following auxiliary problem:
\begin{equation}\label{subso}
\left\{
\begin{array}{ll}
v_t = v_{xx} -  d  v +f(v(t-\tau,x)), & t> t_1,\ x\in (g(t_1), \xi(t)),\\
v(t, \xi(t)) = 0,\quad \xi'(t)= -\mu v_x(t, \xi(t)),& t>t_1,\\
v (t,g(t_1))=0, & t> t_1,\\
\xi(t_1) = h(t_1),\ \  v(s, x)= u(s, x), & s\in[t_1-\tau, t_1],\ x\in [g(s), h(s)].
  \end{array}
 \right.
 \end{equation}
It is easy to check that $v$ is a subsolution of \eqref{p}, then $\xi(t)\leqslant h(t)$ and $\xi(\infty)<\infty$
by our assumption. Using a similar argument as in \cite[Lemma 3.3]{DGP} one can show that
$$
\|v(t,\cdot)- V(\cdot)\|_{C^2([g(t_1),\xi(t)])} \to 0,\quad \mbox{as } t\to\infty,
$$
where $V(x)$ is the unique positive solution of the problem
\[
V''- d  V +f(V)=0\ \ \mbox{ for}\ \ x\in(g(t_1),\xi(\infty)),\ \ \ \ V(g(t_1))=V(\xi(\infty))=0.
\]
Thus,
\[
\lim_{t\to\infty} \xi'(t)=-\mu \lim_{t\to\infty}v_x (t,\xi(t)) =-\mu V'(\xi(\infty)) = \delta,
\]
for some $\delta>0$, which contradicts the fact that $\xi(\infty) < \infty$.

\smallskip

``(ii)$\Rightarrow$(iii)". It follows from the assumption and \cite[Proposition 2.9]{YCW} that the unique positive
solution of the following problem
\begin{equation}\label{upbsoper}
\left\{
\begin{array}{ll}
v_t=v_{xx}- d  v+f(v(t-\tau,x)), & t>0,\ x\in[g_\infty,h_\infty],\\
 v(t,g_\infty)= v(t,h_\infty)=0, &  t>0,\\
 v(\theta,x)\geqslant 0, & \theta\in[-\tau,0],\ x\in[g_\infty,h_\infty],
  \end{array}
 \right.
 \end{equation}
with $v(\theta,x)\geqslant \phi(\theta,x)$ in $[-\tau,0]\times[g(\theta),h(\theta)]$, satisfies $v\to0$ uniformly
for $x\in[g_\infty,h_\infty]$ as $t\to\infty$. Then the conclusion (iii) follows easily from the comparison principle.

\smallskip

``(iii)$\Rightarrow$(ii)": Suppose by way of contraction argument that for some small $\varepsilon>0$ there exists $t_2\gg 1$
such that $h(t)-g(t)>\frac{\pi}{\sqrt{f'(0)- d }}+ 3\varepsilon$ for all $t>t_2-\tau$. Let $l_1:=\pi/\sqrt{f'(0)- d }+ \varepsilon$,
it is well known that the following eigenvalue problem
$$
\left\{
 \begin{array}{ll}
-\varphi_{xx} +  d  \varphi- f'(0)\varphi=\lambda_1\varphi, & 0<x<l_1,\\
 \varphi(0)=\varphi(l_1)=0,
 \end{array}
 \right.
$$
has a negative principal eigenvalue, denoted by $\lambda_1$,
whose corresponding positive eigenfunction, denoted by
$\varphi$, can be chosen positive and normalized by $\|\varphi\|_{L^{\infty}}=1$. Set
\[
w(t,x) :=\epsilon\varphi(x)\ \mbox{ for } x\in[0,l_1],
\]
with $\epsilon>0$ small such that
\[
f(\epsilon\varphi)\geqslant f'(0)\epsilon\varphi+
\frac{1}{2}\lambda_1\epsilon\varphi\ \ \mbox{ in }[0, l_1].
\]
It is easy to compute that for $x\in[0, l_1]$,
$$
w_t-w_{xx}+ d  w-f(w(t-\tau,x))= \epsilon\varphi [f'(0)+\lambda_1 ]
-f(\epsilon\varphi) \leqslant 0.
$$
Moreover one can see that
\[
0\leqslant  w(x) = \epsilon \varphi(x) <  u(t_2+s, x +g(t_2+s)+\varepsilon),\quad
 x\in [0, l_1],\ s\in[-\tau,0]
\]
provided that $\epsilon$ is sufficiently small.
Then we can apply the comparison principle to deduce
$$
u(t+t_2,x +g(t_2) +\varepsilon) \geqslant w(x)>0,\quad (t,x)\in[0,\infty)\times(0, l_1),
$$
contradicting (iii).

\smallskip
``(ii)$\Rightarrow$(i)". When (ii) holds, (i) is obvious. This proves the lemma.
\end{proof}

Next, we give a sufficient condition for vanishing, which indicates that if the initial domain
and initial function are both small, then the species dies out eventually in the environment.

\begin{lem}\label{vfsma}
Assume that {\bf(H)} holds. Let $(u,g,h)$ be a solution of \eqref{p}. Then vanishing happens provided that
$h(0)-g(0)<\frac{\pi}{\sqrt{f'(0)- d }}$ and $\|\phi\|_{L^\infty([-\tau,0]\times[g(\theta),h(\theta)])}$ is sufficient small.
\end{lem}

\begin{proof}
Set
\[
h_0=\frac{h(0)-g(0)}{2},
\]
then $h_0<\pi/(2\sqrt{f'(0)- d })$, so there exists a small $\varepsilon >0$ such that
\begin{equation}\label{choice of delta}
\frac{\pi^2}{4 (1+\varepsilon)^2 h^2_0} - (f'(0)+\varepsilon)e^{\varepsilon\tau} + d \geqslant \varepsilon.
\end{equation}
For such $\varepsilon$, we can find a small positive constant $\delta$ such that
$$
\pi \mu \delta \leqslant \varepsilon^2 h^2_0, \qquad f(v) \leqslant (f'(0) + \varepsilon) v
\quad \mbox{for } v\in [0,\delta].
$$
Define
\begin{align*}
 & k(t) := h_0 \Big( 1+\varepsilon - \frac{\varepsilon}{2} e^{-\varepsilon t}
\Big), \quad w(t,x):= \delta e^{-\varepsilon t} \cos\Big(
\frac{\pi x}{2 k (t)}\Big),\ \ t>0,\ x\in[-k(t),k(t)],\\
 & k(\theta)\equiv k_0 := h_0 \Big( 1+\frac{\varepsilon}{2} \Big), \quad w(\theta,x)\equiv
w_0(x):= \delta \cos\Big(\frac{\pi x}{h_0(2+\varepsilon)}\Big),\ \ \theta\in[-\tau,0],\ x\in[-k_0,k_0].
\end{align*}
and extend $w(t,x)$ by $0$ for $t\in[-\tau,\infty)$, $x\in(-\infty, -k(t)]\cup [k(t),\infty)$.

A direct calculation shows that for $t>0$, $x\in(-k(t),k(t))$
\begin{eqnarray*}
&& w_t - w_{xx} + d  w - f(w(t-\tau,x))\\
&=& \left[ \frac{\pi^2}{4k^2(t)}-\varepsilon + d
-(f'(0)+\varepsilon)\frac{w(t-\tau,x)}{w(t,x)}  +\frac{\pi x k'(t)}{2k^2(t)}\tan \Big(
\frac{\pi x}{2 k (t)}\Big)\right] w\\
& \geqslant& \left[ -\varepsilon + \frac{\pi^2}{4k^2(t)}+ d
-(f'(0)+\varepsilon)\frac{w(t-\tau,x)}{w(t,x)}  \right] w,
\end{eqnarray*}
where we have used $k'(t)>0$, $k(t)>0$ for $t>0$ and $y\tan y\geqslant 0$ for $y\in(-\frac{\pi}{2},\frac{\pi}{2})$.

When $t\geqslant \tau$ and $x\in(-k(t),k(t))$,  it is easy to check that
\begin{eqnarray*}
\mathcal{A} & := & -\varepsilon + \frac{\pi^2}{4k^2(t)}+ d
-(f'(0)+\varepsilon)\frac{w(t-\tau,x)}{w(t,x)} \\
& \geqslant &  -\varepsilon + \frac{\pi^2}{4h_0^2(1+\varepsilon)^2}+ d
-(f'(0)+\varepsilon)e^{\varepsilon\tau} \geqslant 0,
\end{eqnarray*}
where the fact that $\cos\Big(\frac{\pi x}{2 k (t-\tau)}\Big)\leqslant \cos\Big(\frac{\pi x}{2 k (t)}\Big)$ for
$(t,x)\in[\tau,\infty)\times[-k(t),k(t)]$ and the monotonicity of $k(t)$ in $t\in[0,\infty)$ are used. If
$t\in[0, \tau)$ and $x\in(-k(t),k(t))$, we have that
\begin{eqnarray*}
\mathcal{A} &  \geqslant & -\varepsilon + \frac{\pi^2}{4h_0^2(1+\varepsilon)^2}+ d
-(f'(0)+\varepsilon)e^{\varepsilon t}\frac{\cos\Big(\frac{\pi x}{h_0(2+\varepsilon)}\Big)}{\cos\Big(\frac{\pi x}{2k(t)}\Big)} \\
& \geqslant &  -\varepsilon + \frac{\pi^2}{4h_0^2(1+\varepsilon)^2}+ d
-(f'(0)+\varepsilon)e^{\varepsilon\tau} \geqslant 0.
\end{eqnarray*}
Thus we have
$$
w_t - w_{xx} + d  w - f(w(t-\tau,x)) \geqslant 0\ \ \mbox{ in }\ (0,\infty)\times(-k(t),k(t)).
$$

On the other hand,
$$
 k'(t)=\frac{\varepsilon^2 h_0}{2} e^{-\varepsilon t}\geqslant \frac{\pi \mu \delta}{2h_0 } e^{-\varepsilon t}\geqslant
 \frac{\pi \mu \delta}{2k(t)} e^{-\varepsilon t} \geqslant - \mu w_x(t, k(t)) =\mu w_x(t, -k(t)).
$$
As a consequence, $(w(t,x), -k(t), k(t))$ will be a supersolution of \eqref{p} if  $w(\theta,x)\geqslant \phi (\theta,x)$ in $[-\tau,0]\times[g(\theta),h(\theta)]$.
Indeed, choose $\sigma_1 := \delta\cos \frac{\pi}{2+\varepsilon}$, which depends only on $\mu, h_0,  d $ and $f$. Then when $\|\phi\|_{L^\infty([-\tau,0]\times[g(\theta),h(\theta)])}
\leqslant \sigma_1$ we have $\phi(\theta,x)\leqslant \sigma_1 \leqslant w(\theta,x)$ in $[-\tau,0]\times[g(\theta), h(\theta)]$, since $h_0 < k(0)= h_0 (1+\frac{\varepsilon}{2})$.
It follows from the comparison principle that
$$
h(t)\leqslant k(t) \leqslant h_0 (1+\varepsilon),\; h_\infty<\infty.
$$
This, together with the previous lemma, implies that vanishing happens.
\end{proof}
\begin{remark}\rm
When $\tau=0$, the proof of Lemma \ref{vfsma} reduces to that of \cite[Theorem 3.2(i)]{DuLou}.
\end{remark}

We now present a sufficient condition for spreading, which reads as follows.

\begin{lem}\label{lemuto1}
Assume that {\bf(H)} holds. If $h(0)-g(0)\geqslant \pi/\sqrt{f'(0)- d }$, then spreading happens
for every positive solution $(u, g, h)$ of \eqref{p}.
\end{lem}

\begin{proof}
Since $g'(t)<0<h'(t)$ for $t>0$, we have $h(t)-g(t)>\pi/\sqrt{f'(0)- d }$ for any $t>0$. So the conclusion
$-g_\infty = h_\infty =\infty$ follows from Lemma \ref{lemvansmall}. In what follows we prove
\begin{equation}\label{utoPt}
\lim_{t\to\infty}u(t,x)=u^* \mbox{ locally uniformly in $\R$}.
\end{equation}

First, it is well known that for any $L>\pi/(2\sqrt{f'(0)- d })$, the following problem
\[
W_{xx}- d  W+f(W)=0,\ \ \ x\in(-L,L),\ \ \ W(\pm L)= 0,
\]
admits a unique positive solution $W_L$, which is increasing in $L$ and satisfies
\begin{equation}\label{WL1}
\lim_{L\to\infty}W_L(x)=u^* \mbox{ locally uniformly in $\R$}.
\end{equation}
Moreover we can find an increasing sequence of positive numbers $L_n$ with $L_n\to\infty$ as $n\to\infty$ such that
$L_n>\pi/\sqrt{f'(0)- d }$ for all $n\geqslant1$. Since $W_{L_n}$ converges to $u^*$ locally uniformly in $\R$, we
can choose $t_n$ such that $h(t)\geqslant L_n$ and $g(t)\leqslant-L_n$ for $t\geqslant t_n$. It then follows from \cite{YCW}
the following problem
\[
\left\{
\begin{array}{ll}
 w_t =w_{xx}- d  w +f(w(t-\tau,x)), & t\geqslant t_n+\tau,\ x\in[-L_n,L_n],\\
 w(t,\pm L_n)=  0, &  t\geqslant t_n+\tau,\\
 w(s,x)=u(s,x),  & s\in[t_n, t_n+\tau],\ x\in[-L_n,L_n],
\end{array}
\right.
\]
has a unique positive solution $w_n(t,x)$, which satisfies that
\[
w_n(t,x)\to W_{L_n}(x) \ \mbox{ uniformly for } x\in[-L_n,L_n]\ \mbox{ as } t\to\infty.
\]
Applying the comparison principle we have $w_n(t,x)\leqslant  u(t,x)$
for all $t\geqslant t_n+\tau$,  $x\in [-L_n,L_n]$. This, together with \eqref{WL1}, yields that
\begin{equation}\label{uin1}
\liminf_{t\to\infty} u(t ,x) \geqslant  u^*\ \mbox{ locally uniformly for } x\in\R.
\end{equation}
Later, since the initial data $u_0(s,x)$ satisfies $0\leqslant u_0(s,x)\leqslant u^*$ for $(s,x)\in[-\tau,0]\times[g(s),h(s)]$,
it thus follows from the comparison principle that
\[
\limsup_{t\to\infty} u(t ,x) \leqslant  u^*\ \mbox{ locally uniformly for } x\in\R.
\]

Combining with \eqref{uin1}, one can easily obtain \eqref{utoPt}, which ends the proof of this lemma.
\end{proof}

Now we are ready to give the proof of Theorem \ref{thm:asy be}.

\smallskip

\noindent
 {\bf Proof of Theorem \ref{thm:asy be}}. It is easy to see that there are two possibilities: (i) $h_\infty-g_\infty\leqslant \pi/\sqrt{f'(0)- d }$;
 (ii) $h_\infty-g_\infty>\pi/\sqrt{f'(0)- d }$.  In case (i), it follows from Lemma \ref{lemvansmall} that $\lim_{t\to\infty}
\|u(t,\cdot)\|_{L^\infty([g(t),h(t)])}=0$. For case (ii), it follows from Lemma \ref{lemuto1} and its proof that $(g_\infty, h_\infty)=\R$ and
$u(t,x)\to u^*$ as $t\to\infty$ locally uniformly in $\R$, which ends the proof. \hfill
$\square$

\section{Asymptotic profiles of spreading solutions}\label{sec:asybeh}
Throughout this section we assume that {\bf(H)} holds and $(u,g,h)$ is a solution of \eqref{p} for which spreading happens. In
order to determine the spreading speed, we will construct some suitable sub- and supersolutions based on semi-waves.  Let $c^*$ and
$q_{c^*}(z)$ be given in Theorem \ref{waves}.  The first subsection covers the proof of the boundedness for $|h(t)-c^*t|$ and $|g(t)+c^*t|$.
Based on these results, we prove Theorem \ref{thm:profile of spreading sol} in the second subsection.

\subsection{Boundedness for $|h(t)-c^*t|$ and $|g(t)+c^*t|$.}\label{sub51}
Let us begin this subsection with the following estimate.
\begin{lem}\label{lem:u-to-1} Let $(u, g, h)$ be a solution of \eqref{p} for which spreading happens.
Then for any $c\in (0,c^*)$, there exist small $\beta^*\in (0,  d -f'(u^*))$ , $T>0$ and $ M>0$
such that for $t\geqslant T$,
\begin{itemize}
\item[\rm (i)]
$
[g(t), h(t)]\supset [-ct, ct];
$
\item[\rm (ii)]
$ u(t,x)\geqslant u^*\big(1-M e^{-\beta^* t}\big)\quad \mbox{for } x\in [-ct, ct]; $
\item[\rm (iii)]
$ u(t,x) \leqslant u^*\big(1+M e^{-\beta^* t}\big) \quad \mbox{for } x \in [g(t),
h(t)]. $
\end{itemize}
\end{lem}
\begin{proof}
In order to prove conclusions (i) and (ii), inspired by \cite{FM}, we will use the semi-wave $q_{c^*}$ to construct
the suitable subsolution. Here we mainly use the the monotonicity and exponentially convergent of $q_{c^*}$.

(i)\  Since  $q_{c^*}(z)$ is the unique positive solution of
 \begin{equation}\label{semiwave112}
\left\{
 \begin{array}{ll}
 q_{c^*}'' - c^*q_{c^*}'- d  q_{c^*}+ f( q_{c^*}(z-c^*\tau))=0,\ \ \ q_{c^*}'(z)>0, & z>0,\\
 q_{c^*}(z)=0,  & z\leqslant 0,\\
 \mu q_{c^*}'(0)=c^*,\ \ q_{c^*}(\infty)=u^*,
 \end{array}
 \right.
\end{equation}
then it is easy to check that $q_{c^*}''(0)> 0$. Since $q_{c^*}'(z)> 0$ for $z\geqslant0$ and $q_{c^*}(z)\to u^*$ as $z\to\infty$,
thus there is $z_0\gg 1$ such that $q_{c^*}''(z)<0$ for $z\geqslant z_0$. Thus there exists $\hat{z} \in (0,\infty)$ such that
$q_{c^*}''(\hat{z})=0$ and $q_{c^*}''(z)>0 $ for $z\in[0,\hat{z})$. This means that $q_{c^*}'(z)$ is increasing in $z\in[0,\hat{z})$.
Let $\hat{p}_0 \in (0,q_{c^*}(\hat{z}))$ be small. Define
\[
G(u,p)=\left\{
\begin{array}{ll}
 d +[f(u-p)-f(u)]/p ,& p>0 ,\\
 d -f'(u), & p=0,
\end{array} \right.
\]
for $p>0$ and $u>p$. Then $G(u,p)$ is a continuous function for $0 \leqslant p \leqslant \hat{p}_0$ and $G(u^*,p)>0$,
$G(u^*,0)= d -f'(u^*)>0$, thus there exists $0<\gamma\ll  d $ such that $G(u^*,p) \geqslant 2\gamma$ for
$0\leqslant p\leqslant \hat{p}_0 $. By continuity, there exists $\rho>0$ small such that $G(u,p) \geqslant \gamma $ for
$u^*-\rho \leqslant u\leqslant u^*$, $0\leqslant p\leqslant \hat{p}_0$. Furthermore, as $f(u^*)= d  u^*$, then there
is a constant $b>0$ such that
\begin{equation}\label{fub1}
f(v)- d  v\leqslant b(u^*-v)\ \ \mbox{ for }\ v\in[u^*-\rho, u^*].
\end{equation}
Inspired by \cite{FM}, let us construct the following function:
\[
\underline{u}(t,x):= \max\{0,\ q_{c^*}(x+c^*t+\xi(t))+q_{c^*}(c^*t-x+\xi(t))-u^*-p(t)\},\ \ t>0,
\]
and denote $\underline{g}(t)$ and $\underline{h}(t)$ be the zero points of $\underline{u}(t,x)$ with $t>0$,  that is
\[
\underline{u}(t,\underline{g}(t))=\underline{u}(t,\underline{h}(t))=0.
\]

In the following, we will  show that $(\underline{u},\underline{g}, \underline{h})$ is a subsolution of problem \eqref{p}. We only
prove the case where $x\geqslant 0$, since the other is analogous. For any function $J$ depended on $t$, we write $J_{\tau}(t):=J(t-\tau)$
if no confusion arises. For simplicity of notations, we will write
\[
\zeta^-(t):=-x+c^*t+\xi(t), \ \zeta^+(t):=x+c^*t+\xi(t),\ \ \zeta^-_\tau:=\zeta^-(t-\tau), \ \zeta^+_\tau:=\zeta^+(t-\tau).
\]

Firstly, a direct calculation shows that for $(t,x)\in(\tau,\infty)\times[0, \underline{h}(t)]$,
\begin{align*}
\mathcal{N}[\underline{u}]:&=\underline{u}_t-\underline{u}_{xx}+ d  \underline{u}-f(\underline{u}(t-\tau,x))\\
&=\xi'[q'_{c^*}(\zeta^-)+q'_{c^*}(\zeta^+)]+f(q_{c^*}(\zeta^-_\tau))+f(q_{c^*}(\zeta^+_\tau))\\
&\ \ \ -f(q_{c^*}(\zeta^-_\tau)+q_{c^*}(\zeta^+_\tau)-u^*-p_{\tau})- d (u^*+p) -p'.
\end{align*}
Assume that $\xi'(t) \leqslant 0$, and choose $\xi$ large such that  $u^*-\frac{\rho}{2}\leqslant q_{c^*}(\zeta^+_\tau)\leqslant u^*$
in $(\tau,\infty) \times[0, \underline{h}(t)]$. The monotonicity of $q_{c^*}$ and its exponential rate of convergence to $u^*$ at $\infty$
imply that if we choose $\xi$ sufficiently large, then there exist positive constants  $\nu$, $K_0$ and $K$ such that
\[
u^*-q_{c^*}(\zeta^+_\tau)\leqslant K_0e^{-\nu \zeta^+_\tau}\leqslant Ke^{-\nu(\xi(t)+c^*t)}.
\]
Set $p(t)=p_0e^{-\beta t}$ with $p_0:=\frac{1}{2}\min\{\hat{p}_0,\ \frac{\rho}{2}\}$ and $\beta:=\frac{1}{2}\min\{\nu c^*,\ \alpha_0\}$, where
$\alpha_0$ is the unique zero point of
\[
 d  (e^{\tau y}-1)-\gamma e^{\tau y}+y=0.
\]

Thus, when $q_{c^*}(\zeta^-_\tau)\in[u^*-\rho, u^*]$ and $(t,x)\in(\tau,\infty)\times[0, \underline{h}(t)]$, since $q_{c^*}'(z) \geqslant 0$, then
\begin{align*}
\mathcal{N}[\underline{u}]&=\xi'[q'_{c^*}(\zeta^-)+q'_{c^*}(\zeta^+)]+f(q_{c^*}(\zeta^-_\tau))+f(q_{c^*}(\zeta^+_\tau))\\
&\ \ \ -f(q_{c^*}(\zeta^-_\tau)+q_{c^*}(\zeta^+_\tau)-u^*-p_{\tau})- d (u^*+p) -p'\\
&\leqslant \gamma [q_{c^*}(\zeta^+_\tau)-u^*-p_{\tau}]+b[u^*-q_{c^*}(\zeta^+_\tau)]+ d (p_{\tau}-p)-p'\\
&\leqslant b[u^*-q_{c^*}(\zeta^+_\tau)]+ d (p_{\tau}-p)-p'-\gamma p_{\tau}\\
&\leqslant Kbe^{-\nu(\xi(t)+c^*t)}+p_0e^{-\beta t}\big[ d  \big(e^{\beta \tau}-1\big)-\gamma e^{\beta \tau}+\beta\big]\leqslant 0,
\end{align*}
provided that $\xi$ is sufficiently large.

For the part $ q_{c^*}(\zeta^-_\tau)\in[0,u^*-\rho]$, then for $(t,x)\in(\tau,\infty)\times[0, \underline{h}(t)]$ and sufficiently large $\xi$,
there are two positive constants $d_1$ and $d_2$ where $d_1<1$ such that $q'_{c^*}(\zeta^-)+q'_{c^*}(\zeta^+)\geqslant d_1$, and
\[
f\big(q_{c^*}(\zeta^-_\tau)\big)-f\big(q_{c^*}(\zeta^-_\tau)+q_{c^*}(\zeta^+_\tau)-u^*-p_{\tau}\big)+ d [q_{c^*}(\zeta^+_\tau)-u^*-p_{\tau}]\leqslant d_2[u^*+p_{\tau}-q_{c^*}(\zeta^+_\tau)],
\]
thus we have
\begin{align*}
\mathcal{N}[\underline{u}]&=\xi'[q'_{c^*}(\zeta^-)+q'_{c^*}(\zeta^+)]+f(q_{c^*}(\zeta^-_\tau))+f(q_{c^*}(\zeta^+_\tau))\\
&\ \ \ -f(q_{c^*}(\zeta^-_\tau)+q_{c^*}(\zeta^+_\tau)-u^*-p_{\tau})- d (u^*+p) -p'\\
&\leqslant d_1\xi'+d_2 [u^*+p_{\tau}-q_{c^*}(\zeta^+_\tau)]+b[u^*-q_{c^*}(\zeta^+_\tau)+ d (p_{\tau}-p)-p'\\
&\leqslant d_1\xi'+(d_2+b)Ke^{-\nu(\xi+c^*t)} +p_0e^{-\beta t}\big[d_2e^{\beta \tau}+ d \big(e^{\beta \tau}-1\big)+\beta\big]\\
&\leqslant d_1\xi'+p_0e^{-\beta t}\big[d_2e^{\beta \tau}+ d (e^{\beta \tau}-1)+2\beta\big].
\end{align*}
Now let us choose $\xi$ satisfies
\[
d_1\xi'+\kappa p_0e^{-\beta t}=0
\]
with $\xi(0)=\xi_0$ sufficiently large, and $\kappa:=d_2e^{\beta \tau}+ d \big(e^{\beta \tau}-1\big)+2\beta$, then $\xi'(t)\leqslant 0$.
Hence from the above we obtain that $\mathcal{N}[\underline{u}]\leqslant 0$ in this part.

Next, let us check the free boundary condition. When $x=\underline{h}(t)$, we set $\zeta_1(t)=-\underline{h}(t)+c^*t+\xi(t)$ and
$\zeta_2(t)=\underline{h}(t)+c^*t+\xi(t)$, then
\begin{equation}\label{qq1}
q_{c^*}(\zeta_1(t))+q_{c^*}(\zeta_2(t))=u^*+p(t).
\end{equation}
We differentiate \eqref{qq1} with respect to  $t$ to obtain
\begin{equation}\label{hf1}
\big[q_{c^*}'(\zeta_2)-q_{c^*}'(\zeta_1)\big]\big(\underline{h}'(t)-c^*\big)=
p'-2c^*q_{c^*}'(\zeta_2)-\big[q_{c^*}'(\zeta_2)+q_{c^*}'(\zeta_1)\big]\xi'.
\end{equation}
By shrinking $p_0$ and enlarge $\xi_0$ if necessary, then we can see that $\zeta_2(t)\gg1$, and
$q_{c^*}(\zeta_2(t))\approx u^*$. This, together with \eqref{qq1}, yields that $q_{c^*}(\zeta_1(t))\approx p(t)$.
Since $q''_{c^*}(z)>0>q''_{c^*}(y)$ for $0\leqslant z\ll 1$ and $y\gg 1$ and $q'_{c^*}(z)\searrow 0$ as $z\to\infty$, thus we have
\begin{equation}\label{q1q21}
0<q_{c^*}'(\zeta_2)< q_{c^*}'(0)< q_{c^*}'(\zeta_1).
\end{equation}
Thanks to the choice of $\xi(t)$, we can compute that
\begin{equation}\label{q1q22}
p'-2c^*q_{c^*}'(\zeta_2)-[q_{c^*}'(\zeta_2)+q_{c^*}'(\zeta_1)]\xi'\geqslant \big(\frac{\kappa q_{c^*}'(0)}{d_1}-\beta\big) p_0e^{-\beta t}-2c^*K_1e^{-\nu(\xi(t)+c^*t)}\geqslant 0,
\end{equation}
where $K_1$ is a positive constant, $\kappa:=d_2e^{\beta \tau}+ d \big(e^{\beta \tau}-1\big)+2\beta>2\beta$ and we have used that
by shrinking $d_1$ if necessary, then $\kappa q_{c^*}'(0)>\beta d_1$.

It follows from \eqref{hf1}, \eqref{q1q21}, \eqref{q1q22} and the monotonicity of $q_{c^*}'(z)$ in $z$ that
\[
 \underline{h}'(t)\leqslant c^*=\mu q_{c^*}'(0)\leqslant \mu[q_{c^*}'(\zeta_1)-q_{c^*}'(\zeta_2)]=-\mu \underline{u}_x(t,\underline{h}(t)).
\]

Using \eqref{qq1} again, it is easy to see that $\zeta_1(t)$ is decreasing in $t\geqslant T_1$, thus for all $t\geqslant T_1$,
\begin{equation}\label{huh}
\underline{h}(t)-c^*t\geqslant \tilde{C}_0:=\underline{h}(T_1)-c^*T_1+\xi(\infty)-\xi(0).
\end{equation}
Since $(u,g,h)$ is a spreading solution of \eqref{p}, then there exists $T_2>0$ such that
\begin{align*}
 & u(T_1+T_2+\tilde{s},x)\geqslant \underline{u}(T_1+\tau,x)\ \mbox{ for }\ \tilde{s}\in[0,\tau],\ x\in[\underline{g}(\tau),\underline{h}(\tau)],\\
 & g(T_1+T_2)\leqslant \underline{g}(T_1+\tau)\ \ \mbox{and }\ h(T_1+T_2)\geqslant \underline{h}(T_1+\tau).
\end{align*}

Consequently, $(\underline{u},\underline{g}, \underline{h})$ is a subsolution of problem \eqref{p}, then we can apply the comparison
principle to conclude that $u(t+T_1+T_2,x)\geqslant \underline{u}(t+T_1,x)$, $h(t+T_1+T_2)\geqslant \underline{h}(t+T_1)$ for $t>0$, $x\in[0,\underline{h}(t)]$.
This, together with \eqref{huh}, implies that
\[
h(t)-c^*t\geqslant -C_1\ \ \ \mbox{ for } t>0,
\]
with $C_1:=-|\tilde{C}_0|-h(T_1+T_2+\tau)-c^*(T_1+T_2+\tau)$. Similarly, by enlarging $C_1$ if necessary,
we can have $g(t)+c^*t\leqslant C_1$ for $t>0$.
Thus result (i) holds for large $T$.

\smallskip

(ii)\ From the proof of (i), it is easy to see that $u(t+T_2)\geqslant  \underline{u}(t,x)$ for $t>T_1$.
The monotonicity of $q_{c^*}$ and its exponential rate of convergence to
$u^*$ at $\infty$ can be used again to conclude that for any $c\in(0,c^*)$ there exist constants
$\nu$, $K>0$ such that for any $x\in[0,ct]$ and $t>0$,
\begin{align*}
 & u^*-q_{c^*}(x+c^*t+\xi(t))\leqslant u^*-q_{c^*}(c^*t+\xi(t))\leqslant K e^{-\nu(c^*t+\xi(t))},\\
 & q_{c^*}(-x+c^*t+\xi(t))\geqslant q_{c^*}((c^*-c)t+\xi(t))\geqslant u^*-K e^{-\nu[(c^*-c)t+\xi(t)]}.
\end{align*}
Based on above results, we can find $T_3>T_1+T_2$ large such that for $t>T_3$ and $x\in[0,ct]$,
\begin{align*}
u(t,x)&\geqslant q_{c^*}(x+c^*(t-T_2)+\xi(t-T_2))+q_{c^*}(-x+c^*(t-T_2)+\xi(t-T_2))-u^*-p_0e^{\beta (t-T_2)}\\
 & \geqslant u^* -2K e^{-\nu\big[(c^*-c)(t-T_2)+\xi(t-T_2)\big]}-p_0e^{\beta (t-T_2)} \geqslant u^*-M u^*e^{-\beta^* t},
\end{align*}
where $M>0$ is sufficiently large and $\beta^*:=\frac{1}{2}\min\big\{\nu(c^*-c),\ \beta,\  d -f'(u^*)\big\}$.
The case where $x\in[-ct,0]$ can be proved by a similar argument as above. The proof of (ii) is now complete.

\smallskip

(iii)\ Thanks to the choice of the initial data, we know that for any given $\beta^*>0$ and $M>0$,
\[
u(t,x) \leqslant u^*+ Mu^* e^{-\beta^* t}\ \ \ \mbox{ for }\ (t,x)\in[0,\infty)\times[g(t), h(t)].
\]
This completes the proof.
\end{proof}

Next we prove the boundedness of $h(t)-c^*t$ and show that $u(t,\cdot) \approx u^*$ in the domain
$[0, h(t)-Z]$, where $Z>0$ is a large number.

\begin{prop}\label{pro:sigma01}
Assume that spreading happens for the solution $(u,g,h)$. Then
\begin{itemize}
\item[(i)] there exists $C>0$ such that
\begin{equation}\label{hghg1}
|h(t)-c^*t |\leqslant C \ \ \mbox{ for all } t\geqslant0 ;
\end{equation}

\item[(ii)] for any small $\varepsilon>0$, there exists $Z_\varepsilon>0$ and $T_\epsilon >0$ such that
\begin{equation}\label{ughu1}
\|u(t,\cdot ) - u^* \|_{L^\infty ([0, h(t) -Z_\varepsilon])} \leqslant u^*\varepsilon \ \ \mbox{ for } t> T_\varepsilon.
\end{equation}
\end{itemize}
\end{prop}

\begin{proof}
In order to prove conclusions in this proposition, inspired by \cite{DMZ}, we will use the semi-wave $q_{c^*}$ to construct
the suitable sub- and supersolution. Compared with \cite{DMZ}, our problem deal with the case where $\tau>0$. Due to $\tau>0$,
there will be some space-translation of the semi-wave $q_{c^*}$, which make our problem difficult to deal with. To
overcome this difficulty, we mainly use the the monotonicity and exponentially convergent of $q_{c^*}$. Moreover, this idea also
be used in Lemma \ref{limn21}. For clarity we divide the proof into several steps.

\smallskip

$Step\ 1$. To give some upper bounds for $h(t)$ and $u(t,x)$.

Fix $c\in(0,c^*)$. It follows from Lemma \ref{lem:u-to-1} that there exist $\beta^*\in(0,  d -f'(u^*))$, $M >0$, and $T> 0$
such that for $t \geqslant T$, (i), (ii) and (iii) in Lemma \ref{lem:u-to-1} hold. Thanks to {\bf (H)}, by shrinking $\beta^*$
if necessary, we can find $\rho>0$ small such that
\begin{equation}\label{vu1}
 d -f'(v)e^{\beta^* \tau}\geqslant \beta^*\ \ \ \mbox{ for } v\in[u^*-\rho,u^*+\rho].
\end{equation}

For any  $T_*>T+\tau$ large satisfying $Mu^* e^{-\beta^* (T_*-\tau)}<\frac{\rho}{2}$, there is $M' > M$ such that
$M'u^*e^{-\beta^* (T_*-\tau)}< \rho$. Since $q_{c^*}(z)\to u^*$ as $z\to\infty$, we can find $Z_0 >0$ such that
\begin{equation}\label{U1a1}
\big(1+M'e^{-\beta^* (T_*+\tau)  }\big)q_{c^*}(Z_0 )\geqslant u^*.
\end{equation}

Now we construct a supersolution $(\bar{u} ,g, \bar{h})$ to \eqref{p} as follows:
\begin{align*}
 & \bar{h} (t): =c^*(t - T_*)+ h (T_*+\tau )+ K M'\big(e^{-\beta^* T_* }-e^{-\beta^* t}\big)+Z_0\ \ \
\mbox{ for }\ t\geqslant T_* ,\\
 & \bar{u}(t,x):=\min\big\{\big(1+M'e^{-\beta^* t}\big)q_{c^*}\big(\bar{h} (t)-x\big),\ u^*\big\}\ \ \ \mbox{ for }\ t\geqslant T_* ,\ x\leqslant \bar{h} (t),
\end{align*}
where $K$ is a positive constant to be determined below.

Clearly, for all $t\geqslant T_*$, $\bar{u} (t, g(t))>0= u(t, g(t))$, $\bar{u}\big(t, \bar{h} (t)\big)=0$, and
\begin{eqnarray*}
-\mu \bar{u} _x(t,\bar{h}(t))& = & \mu \big(1+M'e^{-\beta^* t}\big)q_{c^*}'(0)=\big(1+M'e^{-\beta^* t}\big)c^*, \\
& < & c^*+M' K \beta^* e^{-\beta^* t} = \bar{h}'(t),
\end{eqnarray*}
if we choose $K$ with $K\beta^* > c^*$.  By the definition of $\bar{h}$ we have $h (T_*+s )<\bar{h}(T_* +s)$ for $s\in[0,\tau]$. It then follows
from  \eqref{U1a1} that for $(s,x)\in[0,\tau]\times[ g (T_*+s ),h (T_* +s)]$,
\[
\big(1+M'e^{-\beta^* (T_*+s) }\big)q_{c^*}\big(\bar{h} (T_*+s)-x\big) \geqslant\big(1+M'e^{-\beta^* (T_*+\tau) }\big)q_{c^*}(Z_0)\geqslant u^*,
\]
which yields that $\bar{u}(T_*+s,x)=u^*\geqslant u(T_*+s,x)$ for $(s,x)\in[0,\tau]\times[g(T_*+s),h(T_*+s)]$.

We now show that
\begin{equation}\label{u+ upper}
 \mathcal{N} [\bar{u}] := \bar{u}_t - \bar{u}_{xx} + d  \bar{u}-f(\bar{u}(t-\tau,x)) \geqslant 0,\quad x\in [g(t), \bar{h}(t)],\ t> T_*+\tau .
\end{equation}
Thanks to the definition of $\bar{u}(t,x)$ and the monotonicity of $q_{c^*}(z)$ in $z$, we can find a decreasing function
$\eta(t)<\bar{h}(t)$ for $t>T_*$, such that
\[
\big(1+M'e^{-\beta^* t}\big)q_{c^*}\big(\bar{h}(t)-x\big)\left\{\begin{array}{ll} > u^*, &  x<\eta(t),\\
= u^*, &  x=\eta(t),\\
< u^*, & x\in\big(\eta(t),\bar{h}(t)\big],
\end{array}
\right.
\]
which implies that
\[
\bar{u}(t,x)=u^*\ \mbox{ for } x \leqslant \eta(t), \ \mbox{ and }\  \bar{u}(t,x)=\big(1+M'e^{-\beta^* t}\big)q_{c^*}\big(\bar{h}(t)-x\big)\ \mbox{ for } x\in\big[\eta(t),\bar{h}(t)\big].
\]
As $\mathcal{N} u^*=0$, thus in what follows, we only consider the case $ x\in\big[\eta(t),\bar{h}(t)\big]$. Set $q_\tau:=q_{c^*}\big(\bar{h}_\tau-x\big)$ for convenience.
A direct calculation shows that, for $t>T_*+\tau$,
\begin{align*}
\mathcal{N} [\bar{u}] :& = \bar{u}_t - \bar{u}_{xx} + d  \bar{u}-f(\bar{u}(t-\tau,x))\\
 & = -\beta^* M'e^{-\beta^* t} q_{c^*}+\big(1+M'e^{-\beta^* t}\big) \{K\beta^* M'e^{-\beta^* t} q_{c^*}'+f(q_{\tau})\} -f\big((1+M'e^{-\beta^* (t-\tau)})q_{\tau}\big)\\
 & = M'e^{-\beta^* t}\Big\{f(q_{\tau})+K \beta^*\big(1+M'e^{-\beta^* t}\big)q_{c^*}' -\beta^*q_{c^*} \Big\} + f(q_\tau)- f\big((1+M'e^{-\beta^* (t-\tau)})q_\tau\big)\\
 & \geqslant M'e^{-\beta^* t}\Big\{K \beta^* \big(1+M'e^{-\beta^* t}\big)q_{c^*}'-\big[\big(f'\big((1+ \theta M'e^{-\beta^* (t-\tau)})q_\tau\big)e^{\beta^* \tau}- d  \big)q_\tau-\beta^*q_{c^*}\big]\Big\},
\end{align*}
for some $\theta \in (0,1)$. Since
\begin{equation}\label{qcto1}
q_{c^*}(z)\to u^*\ \mbox{ and } \frac{(q_{c^*}(z)-u^*)'}{q_{c^*}(z)-u^*}\to k^*\ \ \mbox{ as } z\to \infty
\end{equation}
where $k^*:=c^*-\sqrt{(c^*)^2+4( d -f'(u^*))}<0$, there are $z_0>0$  and $k_1>0$ such that
\begin{equation}\label{qqq1}
q_{c^*}''(z)<0,\ \ \ q_{c^*}(z)\geqslant u^*-\rho\ \ \mbox{ and }\ \ q_{c^*}'(z-2c^*\tau)
\leqslant  k_1 q_{c^*}'(z)  \ \mbox{ for } \ z>z_0,
\end{equation}
Moreover, we can compute that
\begin{align*}
\triangle \bar{h}(t) := \bar{h}(t)-\bar{h}_\tau(t)= c^*\tau+KM'e^{-\beta^*t}(e^{\beta^*\tau}-1).
\end{align*}
For any given $K>0$, by enlarging $T_*$ if necessary,  we have that
\begin{equation}\label{deh}
\triangle \bar{h}(t)\in[c^*\tau,2c^*\tau]\ \ \mbox{ for }\ t \geqslant T_*.
\end{equation}

When $\bar{h}_\tau-x>z_0$ and $t> T_*+\tau$, it then follows that
\begin{align*}
\mathcal{B}  :& = K \beta^* \big(1+M'e^{-\beta^* t}\big)q_{c^*}'-\big[\big(f'\big((1+ \theta M'e^{-\beta^* (t-\tau)})q_\tau\big)e^{\beta^* \tau}- d  \big)q_\tau-\beta^*q_{c^*}\big]\\
 & \geqslant \big[ d  -f'\big(\big(1+ \theta M'e^{-\beta^* (t-\tau)}\big)q_\tau\big)e^{\beta^* \tau}-\beta^*\big]q_\tau+
K\beta^* q_{c^*}'+\beta^* (q_\tau-q_{c^*})\\
& \geqslant K\beta^* q_{c^*}'(\bar{h}(t)-x)-\beta^* q_{c^*}'(\bar{h}(t)-x-\tilde{\theta}\triangle\bar{h}(t))\triangle\bar{h}(t)\ \ \ (\mbox{with } \tilde{\theta}\in(0,1))\\
& \geqslant (K-2k_1c^*\tau) \beta^*q_{c^*}'(\bar{h}(t)-x)\geqslant 0
\end{align*}
provided that $K$ is sufficiently large, and we have used $M'e^{-\beta^* (t-\tau)} u^*\leqslant\rho$ for $t> T_* $,  $q_{c^*}'(z)>0$ for
$z>0$, \eqref{vu1}, \eqref{qqq1} and \eqref{deh}. Thus $\mathcal{N} [\bar{u}]\geqslant 0$ in this case.

When  $0\leqslant \bar{h}_\tau-x\leqslant z_0$ and $t> T_*+\tau $, for sufficiently large $K$,  we have
$$
\mathcal{N}[\bar{u}] \geqslant  M'e^{-\beta^* t}\big[K \beta^* D_1 - D_2u^* e^{\beta^* \tau}-\beta^*u^*\big]\geqslant 0,
$$
where $D_1:=\min_{z\in[0,z_0+2c^*\tau]}q_{c^*}'(z)>0$, $D_2:=\max_{v\in[0,2u^*]}f'( v)$, and \eqref{deh} are used.

Summarizing the above results we see that $(\bar{u}, g, \bar{h})$ is a supersolution of \eqref{p}.  Thus we can apply the comparison
principle to deduce
$$
h(t) \leqslant \bar{h}(t) \quad \mbox{and} \quad  u(t,x)\leqslant \bar{u}(t,x) \leqslant u^*+M'
u^*e^{-\beta^* t}\quad \mbox{ for } x\in [g(t), h(t)],\ t>T_*.
$$
By the definition of $\bar{h}$ we see that, for $C_r := h(T_*+\tau)+Z_0 +KM'$, we have
\begin{equation}\label{hbd}
h(t)< c^*t +C_r\ \ \ \mbox{ for all } t\geqslant 0.
\end{equation}
For any $\varepsilon>0$, if we choose $T_1(\varepsilon) >T_*$ large such that
$M' e^{-\beta^* T_1(\varepsilon)} < \varepsilon$, then we have
\begin{equation}\label{v<1+epsilon/P}
u(t,x)\leqslant \bar{u}(t,x) \leqslant u^*(1 +\varepsilon) ,\quad x\in [g(t), h(t)],\ t> T_1(\varepsilon),
\end{equation}
which ends the proof of Step 1.

\smallskip

$Step\ 2$. To give some lower bounds for $h(t)$ and $u(t,x)$.

Let $c$, $M$, $T$ and  $\beta^*$ be as before. By shrinking $c$ if necessary, we can find $T^*>T+\tau$ large such that
\begin{equation}\label{vu2}
M u^* e^{-\beta^* (t-\tau)}\leqslant \frac{\rho}{2}\ \ \ \ \mbox{ for }\ t\geqslant T^*\ \mbox{ and }\ \ \ h(T^*)-cT^*\geqslant c^*\tau.
\end{equation}
We will define the following functions
\begin{align*}
 & \underline{g}(t)=ct,\ \ \underline{h}(t)=c^*(t-T^*)+cT^*-\sigma M(e^{-\beta^*T^*}-e^{-\beta^*t}),\ \ \ t\geqslant T^*,\\
 & \underline{u}(t,x)=\big(1-Me^{-\beta^* t}\big)q_{c^*}(\underline{h}(t)-x),\ \ \ t\geqslant T^*,\ \ x\in[\underline{g}(t),\underline{h}(t)],
\end{align*}
where $\sigma$ is a positive constant to be determined later.

We will prove that $(\underline{u},\underline{g},\underline{h})$ is a subsolution to \eqref{p} for $t>T^*$. Firstly, for $t\geqslant T^*$,
\[
\underline{u}\big(t,\underline{g}(t)\big)=\underline{u}(t,-ct)\leqslant u^*-M u^* e^{-\beta^* t}\leqslant u(t,-ct)=u\big(t,\underline{g}(t)\big).
\]
Next, we check that $\underline{h}$ and $\underline{u}$ satisfy the required conditions at $x=\underline{h}(t)$. It is obvious that
$\underline{u}(t,\underline{h}(t))=0$. If we choose $\sigma$ with $\sigma\beta^*\geqslant c^*$, then
\begin{eqnarray*}
-\mu \underline{u}_x(t,\underline{h}(t))& = & \mu\big(1-Me^{-\beta^* t}\big)q_{c^*}'(0)=c^*\big(1-Me^{-\beta^* t}\big), \\
& > & c^*-\sigma M \beta^* e^{-\beta^* t} = \underline{h}'(t).
\end{eqnarray*}

Later, let us check the initial conditions. From Lemma \ref{lem:u-to-1}, it is easy to see that
\begin{align*}
&\underline{h}(T^*+s)\leqslant cT^*+c^*\tau \leqslant h(T^*+s),\\
&\underline{u}(T^*+s,x)\leqslant u^*\big(1-Me^{-\beta^* (T^*+s)}\big)\leqslant u(T^*+s,x),
\end{align*}
for $s\in[0,\tau]$ and $x\in [\underline{g}(T^*+s),\underline{h}(T^*+s)]$.

Finally we will prove that $\underline{u}_t-\underline{u}_{xx}+ d \underline{u}-f(\underline{u}(t-\tau,x))\leqslant 0$
for $t\geqslant T^*+\tau$. Put  $z=\underline{h}(t)-x$ and $q_\tau=q_{c^*}(\underline{h}(t-\tau)-x)$.
It is easy to check that
\begin{align*}
\mathcal{N}[\underline{u}]:&=\underline{u}_t-\underline{u}_{xx}+ d  \underline{u}-f(\underline{u}(t-\tau,x))\\
&\leqslant M e^{-\beta^* t}\Big\{\beta^*q_{c^*}-\sigma \beta^*\big(1-M e^{-\beta^* t}\big)q_{c^*}'+\big[f'\big(\big(1-\theta_1M e^{-\beta^* (t-\tau)}\big)q_\tau\big)e^{\beta^*\tau}- d \big]q_\tau\Big\}.
\end{align*}
for some $\theta_1\in(0,1)$. It follows from \eqref{qcto1} that
there are two constants $z_1>0$, $k_2>0$ such that
\begin{equation}\label{qqq12}
q_{c^*}''(z)<0,\ \ \ q_{c^*}(z)\geqslant u^*-\frac{\rho}{2}\ \ \mbox{ and }\ \ q_{c^*}'(z-c^*\tau)
\leqslant  k_2 q_{c^*}'(z)  \ \mbox{ for } \ z>z_1,
\end{equation}

Moreover, we can compute that
\begin{align*}
\triangle \underline{h}(t) : = \underline{h}(t)-\underline{h}_\tau(t) = c^*\tau-\sigma M e^{-\beta^*t}(e^{\beta^*\tau}-1).
\end{align*}
For any given $\sigma>0$, by enlarging $T^*$ if necessary,  we have that
\begin{equation}\label{deh2}
\triangle \underline{h}(t)\in[0,c^*\tau]\ \ \mbox{ for }\ t \geqslant T^*.
\end{equation}
When $\underline{h}_\tau-x>z_1$ and $t\geqslant T^*+\tau$, it then follows that
\begin{align*}
\mathcal{C}  :& = \beta^*q_{c^*}-\sigma \beta^*\big(1-M e^{-\beta^* t}\big)q_{c^*}'+\big[f'\big(\big(1-\theta_1M e^{-\beta^* (t-\tau)}\big)q_\tau\big)e^{\beta^*\tau}- d \big]q_\tau\\
 & \leqslant \big[f'\big(\big(1-\theta_1M e^{-\beta^* (t-\tau)}\big)q_\tau\big)e^{\beta^*\tau}- d +\beta^*\big]q_\tau
-\sigma\beta^* q_{c^*}'+\beta^*(q_{c^*}-q_{\tau})\\
& \leqslant -\sigma\beta^* q_{c^*}'(\underline{h}(t)-x)+\beta^* q_{c^*}'(\underline{h}(t)-x-\tilde{\theta}_1\triangle \underline{h}(t))\triangle \underline{h}(t)
\ \ \ (\mbox{with } \tilde{\theta}_1\in(0,1))\\
& \leqslant (k_2c^*\tau-\sigma) \beta^*q_{c^*}'(\underline{h}(t)-x)\leqslant 0
\end{align*}
provided that $\sigma$ is sufficiently large, and we have used $\big(1-\theta_1 M e^{-\beta^* (t-\tau)}\big)q_\tau\in[u^*-\rho,u^*]$ and
\eqref{vu2} for $t\geqslant T^*$,  and \eqref{vu1}, \eqref{qqq12}, \eqref{deh2}. Thus $\mathcal{N} [\underline{u}]\leqslant 0$ in this case.

When  $0\leqslant \underline{h}_\tau-x\leqslant z_1$ and $t\geqslant T^*+\tau $, for
sufficiently large $\sigma$, we have
$$
\mathcal{N} [\underline{u}] \leqslant M e^{-\beta^*t} \Big[\beta^*u^*-\sigma \beta^* \Big(1-\frac{\rho}{2u^*} e^{-\beta^*\tau}\Big)D'_1 + D'_2u^* e^{\beta^* \tau}\Big]\leqslant 0,
$$
where $D'_1:=\min_{z\in[0,z_1+c^*\tau]}q_{c^*}'(z)>0$, $D'_2:=\max_{v\in[0,2u^*]}f'( v)$ and \eqref{deh2} are used.

Consequently, $(\underline{u},\underline{g}, \underline{h})$ is a subsolution to \eqref{p}, then the comparison principle implies that
\[
\underline{h}(t)\leqslant h(t),\ \ \underline{u}(t,x)\leqslant u(t,x)\ \ \mbox{ for }\ t\geqslant T^*,\ x\in[\underline{g}(t), \underline{h}(t)],
\]
which yields that
\begin{equation}\label{hbd2}
h(t)\geqslant \underline{h}(t) - \max_{t\in[0,T^*]}|h(t)-\underline{h}(t)| \geqslant c^*t -C_l \ \ \mbox{ for all } t\geqslant 0,
\end{equation}
where $C_l = \max_{t\in[0,T^*]}|h(t)-\underline{h}(t)|+c^*T^* +\sigma M$.  Combining with \eqref{hbd} we obtain \eqref{hghg1}.

On the other hand, for any $\varepsilon>0$, since $q_{c^*}(\infty) =u^*$, there exists $Z_1(\varepsilon)>0$ such that
$$
q_{c^*}(z)> u^*\Big(1- \frac{\varepsilon}{2}\Big)\ \  \mbox{  for }  z\geqslant Z_1(\varepsilon).
$$
It follows from \eqref{hbd2} and \eqref{hbd} that
$$
\underline{h}(t) -x \geqslant c^*t -C_l -x \geqslant h(t) - C_r -C_l  -x \geqslant Z_1(\varepsilon)\ \mbox{ for }\ t>T^*,
$$
which yields that for $(t,x)\in \Phi_1 := \{ (t,x) : ct\leqslant x\leqslant h(t) -C_r -C_l -Z_1(\varepsilon),\ t>T^*\}$,
$$
u(t,x) \geqslant \underline{u} (t,x) \geqslant \big(1-M e^{-\beta^* t} \big) q_{c^*}\big(Z_1(\varepsilon)\big) \geqslant
u^*\big(1-M e^{-\beta^* t} \big) \Big( 1- \frac{\varepsilon}{2} \Big).
$$
Moreover, if we choose $T_2(\epsilon) >T^*$ such that $2 M e^{-\beta^* T_2(\varepsilon)} <\varepsilon$, then
\begin{equation}\label{v>1-epsilonP}
u(t,x)\geqslant u^*\Big( 1- \frac{\varepsilon}{2}\Big)^2 > u^*(1- \varepsilon)\ \  \mbox{ for }
(t,x)\in \Phi_1\  \mbox{ and }\ t>T_2(\varepsilon),
\end{equation}
which completes the proof of Step 2.

\smallskip

$ Step\ 3$. Completion of the proof of \eqref{ughu1}.
Denote $T_\varepsilon :=T_1(\varepsilon)+T_2(\varepsilon)$ and $Z_\varepsilon := C_r + C_l +Z_1(\varepsilon)$, then
by \eqref{v<1+epsilon/P} and \eqref{v>1-epsilonP} we have
$$
|u(t,x)-u^*| \leqslant u^*\varepsilon\ \  \mbox{ for } 0\leqslant x\leqslant h(t) -Z_\varepsilon,\ t>T_\varepsilon.
$$
This yields the estimate in \eqref{ughu1}, which completes the proof of this proposition.
\end{proof}

Using a similar argument as above we can obtain the following result.
\begin{prop}\label{pro:sigma12}
Assume that spreading happens for the solution $(u,g,h)$. Then
\begin{itemize}
\item[(i)] there exists $C'>0$ such that
\begin{equation}\label{hhgg1}
|g(t)+c^*t |\leqslant C' \ \ \mbox{for all } t\geqslant0 ;
\end{equation}

\item[(ii)] for any small $\varepsilon>0$, there exists $Z'_\varepsilon>0$ and $T'_\epsilon >0$ such that
\begin{equation}\label{uugghh1}
\|u(t,\cdot ) - u^* \|_{L^\infty ([g(t)+Z'_\varepsilon,0])} \leqslant u^*\varepsilon \ \ \mbox{ for } t> T'_\varepsilon.
\end{equation}
\end{itemize}
\end{prop}

\subsection{Asymptotic profiles of the spreading solutions}\label{sub52}
This subsection is devoted to the proof of Theorem \ref{thm:profile of spreading sol}. We will prove this theorem by a
series of results. Firstly, it follows from Proposition \ref{pro:sigma01} that there exist positive constant $C$ such that
\[
-C\leqslant h(t)-c^*t\leqslant C\ \ \mbox{ for } t\geqslant 0.
\]
Let us use the moving coordinate $y :=x-c^*t+2C$ and set
$$
\begin{array}{l}
h_1(t):= h(t)-c^*t+2C, \quad g_1(t) :=g(t)-c^*t+2C\ \  \mbox{ for } t\geqslant 0,\\
\mbox{and }  u_1(t,y):= u (t, y+c^*t-2C)\ \  \mbox{ for } y\in[g_1(t),h_1(t)],\ t\geqslant 0.
\end{array}
$$
Then $({u_1},{g_1},{h_1})$ solves
\begin{equation}\label{pWH}
\left\{
\begin{array}{ll}
 (u_1)_t =(u_1)_{yy}+c^*(u_1)_y- d  u_1+ f(u_1(t-\tau,y+c^*\tau)), &  {g_1}(t)<y<{h_1}(t),\ t>0,\\
 {u_1}(t, y)= 0,\ {g'_1}(t)=-\mu (u_1)_y(t,y)-c^*, &  y={g_1}(t),\ t>0,\\
 {u_1}(t, y)= 0,\ {h'_1}(t)=-\mu (u_1)_y(t,y)-c^*, &  y={h_1}(t),\ t>0.
\end{array}
\right.
\end{equation}
Let $t_n\to\infty$ be an arbitrary sequence satisfying $t_n>\tau$ for $n\geqslant1$. Define
\[
v_n(t,y)=u_1(t+t_n,y),\ \ H_n(t)=h_1(t+t_n), \ \ k_n(t)=g_1(t+t_n).
\]
\begin{lem}\label{limn1}
Subject to a subsequence,
\begin{equation}\label{vhgt1}
H_n(t)\to H\ \  in \ C^{1+\frac{\nu}{2}}_{loc}(\R)\ \ \ and \ \ \ \|v_n-V\|_{C^{\frac{1+\nu}{2},1+ \nu}_{loc}(\Omega_n)}\to 0,
\end{equation}
where $\nu\in(0,1)$, $\Omega_n=\{(t,y)\in\Omega: \ y\leqslant H_n(t)\}$, $\Omega=\{(t,y): \ -\infty<y\leqslant H(t), t\in\R\}$,
and $(V(t,y),H(t))$ satisfies
\begin{equation}\label{VGHQ}
\left\{
\begin{array}{ll}
 V_t =V_{yy}+c^*V_{y}- d  V+ f(V(t-\tau,y+c^*\tau)), &   (t,y)\in\Omega,\\
 V (t, H(t))= 0,\ H'(t)=-\mu V_{y} (t,H(t))-c^*, & t\in\R.
\end{array}
\right.
\end{equation}
\end{lem}
\begin{proof}
It follows from the proof of Lemma \ref{lem:global} that there is $C_0>0$ such that $0<h'(t)\leqslant C_0$ for all $t>0$. One can deduce that
\[
-c^*<H_n'(t)\leqslant C_0 \ \ \mbox{ for } t+t_n\ \mbox{ large and every } n\geqslant1.
\]
Define
\[
z=\frac{y}{H_n(t)},\ \ \ w_n(t,z)=v_n(t,y),
\]
and direct computations yield that
\[(w_n)_t =\frac{1}{H^2_n(t)}(w_n)_{zz}+\frac{c^*+zH'_n(t)}{H_n(t)}(w_n)_z- d  w_n+ f\Big(w_n\Big(t-\tau,\frac{H_n(t)z+c^*\tau}{H_n(t-\tau)}\Big)\Big)\]
for $\frac{k_n(t)}{H_n(t)} <z<1$, $t>\tau-t_n$, and
\[
w_n(t,1)=0,\ \ H_n'(t)=-\mu \frac{(w_n)_z(t,1)}{H_n(t)}-c^*,\ \ t>\tau-t_n.
\]
Since $w_n\leqslant u^*$, then $f\Big(w_n\Big(t-\tau,\frac{H_n(t)z+c^*\tau}{H_n(t-\tau)}\Big)\Big)$ is bounded. For any given
$Z>0$ and $T_0\in\R$, using the partial interior-boundary $L^p$ estimates and the Sobolev embedding theorem (see \cite{DMZ2, Fr}),
for any $\nu'\in(0,1)$, we obtain
\[
\|w_n\|_{C^{\frac{1+\nu'}{2},1+ \nu'}([T_0,\infty)\times[-Z,1])}\leqslant C_Z\ \ \mbox{ for all large } n,
\]
where $C_Z$ is a positive constant depending on $Z$ and $\nu'$ but independent of $n$ and $T_0$. Thanks to this, we have
\[
\|H_n\|_{C^{1+\frac{\nu'}{2}}([T_0,\infty))}\leqslant C_1\ \ \mbox{ for all large } n,
\]
with $C_1$ is a positive constant independent of $n$ and $T_0$.  Hence by passing to a subsequence we may assume that as $n\to\infty$,
\[
w_n\to W\ \ \mbox{ in } C_{loc}^{\frac{1+\nu}{2},1+ \nu}(\R\times(-\infty,1]),\ \ \ H_n\to H\ \ \mbox{ in } C_{loc}^{1+\frac{\nu}{2}}(\R),
\]
where $\nu\in(0,\nu')$. Based on above results, we can see that $(W,H)$ satisfies that
\[
\left\{
\begin{array}{ll}
 W_t =\frac{W_{zz}}{H^2(t)}+\frac{c^*+zH'(t)}{H(t)}W_{z}- d  W+ f(W(t-\tau,H(t)z+c^*\tau)),  & (t,z)\in(-\infty,1]\times\R,\\
 W (t, 1)= 0,\ \ \ \ H'(t)=-\mu \frac{W_{z} (t,1)}{H(t)}-c^*,  &  t\in\R.
\end{array}
\right.
\]
Define $V(t,y)=W\big(t, \frac{y}{H(t)}\big)$. It is easy to check that $(V,H)$ satisfies \eqref{VGHQ} and \eqref{vhgt1} holds.
\end{proof}

Later, we show by a sequence of lemmas that $H(t)\equiv H_0$ is a constant and hence
\[
V(t,y)=q_{c^*}(H_0-y).
\]

Since $C\leqslant h(t)-c^*t+2C\leqslant 3C$ for all $t\geqslant0$, then $C\leqslant H(t)\leqslant 3C$ for $t\in\R$. Denote
\[
\phi(z):=q_{c^*}(-z)\ \ \mbox{ for } z\in\R,
\]
it follows from the proof of Proposition \ref{pro:sigma01} that for $x\in[(c-c^*)(t+t_n),H_n(t)]$ and $t+t_n$ large,
\[
 \big(1-Me^{-\beta^* (t+t_n)}\big)\phi(y-C)\leqslant v_n(t,y)\leqslant \min\Big\{\big(1+M'e^{-\beta^* (t+t_n)}\big)\phi(y-3C),\ u^*\Big\}.
\]
Letting $n\to\infty$ we have
\[
\phi(y-C)\leqslant V(t,y)\leqslant  \phi(y-3C) \ \ \ \mbox{for all } t\in\R,\ y<H(t).
\]
Define
\[
X^*:=\inf\{X:\ V(t,y)\leqslant  \phi(y-X)\ \ \mbox{ for all } (t,y)\in D\}
\]
and
\[
X_*:=\sup\{X:\ V(t,y)\geqslant  \phi(y-X)\ \ \mbox{ for all } (t,y)\in D\}
\]
Then
\[\phi(y-X_*)\leqslant V(t,y)\leqslant  \phi(y-X^*)\ \ \mbox{ for all } (t,y)\in D,\]
and
\[
C\leqslant X_*\leqslant\inf_{t\in\R} H(t)\leqslant \sup_{t\in\R} H(t)\leqslant X^*\leqslant 3C.
\]
By a similar argument as in \cite{DMZ2}, we have the following result.
\begin{lem}\label{limn2}
$X^*=\sup_{t\in\R} H(t)$, $X_*=\inf_{t\in\R} H(t)$, and there exist two sequences $\{s_n\}$, $\{\tilde{s}_n\}\subset \R$
such that
\[
H(t+s_n)\to X^*,\ \ V(t+s_n,y)\to \phi(y-X^*)\ \ \mbox{ as } n\to\infty
\]
uniformly for $(t,y)$ in compact subsets of $\R\times(-\infty,X^*]$, and
\[
H(t+\tilde{s}_n)\to X_*,\ \ V(t+\tilde{s}_n,y)\to \phi(y-X_*)\ \ \mbox{ as } n\to\infty
\]
uniformly for $(t,y)$ in compact subsets of $\R\times(-\infty, X_*]$.
\end{lem}
Based on Lemma \ref{limn2}, we have the following lemma.
\begin{lem}\label{limn21}
$X^*=X_*$, and hence $H(t)\equiv H_0$ is a constant, which yields $V(t,y)=\phi (y-H_0)$.
\end{lem}
\begin{proof}
Argue indirectly we may assume that $X_*<X^*$. Choose $\epsilon=(X^*-X_*)/4$. We will show next that there is
$T_\epsilon>0$ such that
\begin{equation}\label{HGH1}
H(t)-X^*\geqslant -\epsilon\ \ \mbox{ and }\ \  H(t)-X_*\leqslant \epsilon\ \ \mbox{ for } t\geqslant T_\epsilon,
\end{equation}
which implies that $X^*-X_*\leqslant 2\epsilon$. This contraction would complete the proof.

To complete the proof, we need to prove that for given $\epsilon=(X^*-X_*)/4$, there exist $n_1(\epsilon)$ and $n_2(\epsilon)$
such that
\[
H(t)-X^*\geqslant -\epsilon\ \ (\forall t\geqslant s_{n_1}),\ \ \ H(t)-X_*\leqslant \epsilon\ \ (\forall t\geqslant \tilde{s}_{n_2}).
\]
It follows from $\phi(y-X_*)\leqslant V(t,y)\leqslant \phi(y-X^*)$
that there exist $C_1>0$ and $\beta_1>0$ such that
\[
|u^*-V(t,y)|\leqslant C_1e^{\beta_1y}.
\]
By Lemma \ref{limn2}, for any $\varepsilon>0$, there exist $K>0$, $T>0$ such that for $\tilde{s}_n>T+\tau$ and $s\in[0,\tau]$,
\begin{equation}\label{UHG1}
\sup_{y\in(-\infty,K]}|V(\tilde{s}_n+s,y)-\phi(y-X^*)|<\varepsilon
\end{equation}
 Set $G(t)=H(t)+c^*t$ and $U(t,y)=V(t,y-c^*t)$, then $(W,G)$ satisfies
\begin{equation}\label{UGu}
\left\{
\begin{array}{ll}
 U_t =U_{yy}- d  U+ f(U(t-\tau,y)),  & t\in\R,\ y\leqslant G(t),\\
 U (t, G(t))= 0,\ \ \ \ G'(t)=-\mu U_y (t,G(t)),  &  t\in\R.
\end{array}
\right.
\end{equation}
It follows from Lemma \ref{limn2} and \eqref{UHG1} that there is $n_1=n_1(\varepsilon)$ such that for $n\geqslant n_1$,
\begin{eqnarray}
& H(\tilde{s}_n+s)\leqslant X_*+\varepsilon\ \ \mbox{ for } s\in[0,\tau], \label{GR1}\\
& V(\tilde{s}_n+s,y)\leqslant \phi(y-X_*-\varepsilon)+\varepsilon\ \ \mbox{ for } s\in[0,\tau],\ y\leqslant X_*. \label{UR1}
\end{eqnarray}

Thanks to {\bf (H)}, for $\beta_0\in(0,\beta^*)$ small with $\beta^*$ is given in the proof of Proposition \ref{pro:sigma01},
there is $\eta>0$ small such that
\begin{equation}\label{vuf1}
 d -f'(v)e^{\beta_0 \tau}\geqslant \beta_0\ \ \ \mbox{ for } v\in[u^*-\eta,u^*+\eta],
\end{equation}
and we can find $N>1$ independent of $\varepsilon$ satisfies
\[
 \phi(y-X_*-\varepsilon)+\varepsilon \leqslant \big(1+N\varepsilon e^{-\beta_0\tau}\big)\phi(y-X_*-N\varepsilon)\ \
 \mbox{ for } y\leqslant X_*+\varepsilon,
\]

Let us construct the following supersolution of problem \eqref{UGu}:
$$
\begin{array}{l}
\bar{G}(t):= X_*+N\varepsilon+c^*t+N\sigma\varepsilon\big(1-e^{-\beta_0(t-\tilde{s}_n)}\big),\\
\bar{U}(t,y):=\min\big\{\big(1+N\varepsilon e^{-\beta_0(t-\tilde{s}_n)}\big)\phi\big(y-\bar{G}(t)\big),\ u^*\big\}.
\end{array}
$$
Since $\lim_{y\to-\infty}\big(1+N\varepsilon e^{-\beta_0(t-\tilde{s}_n)}\big)\phi\big(y-\bar{G}(t)\big)>u^*$, then there is a
smooth function $\bar{K}(t)$ of $t\geqslant \tilde{s}_n$ such that $\bar{K}(t)\to-\infty$ as $t\to\infty$ and
$\big(1+N\varepsilon e^{-\beta_0(t-\tilde{s}_n)}\big)\phi\big(\bar{K}(t)-\bar{G}(t)\big)=u^*$. We will check that
$(\bar{U},\bar{K},\bar{G})$ is a supersolution for $t\geqslant \tilde{s}_n+\tau$ and $y\in[\bar{K}(t),\bar{G}(t)]$. We note that
\[
\bar{U}(t,y)=\big(1+N\varepsilon e^{-\beta_0(t-\tilde{s}_n)}\big)\phi\big(y-\bar{G}(t)\big)\ \mbox{ when }\ y\in[\bar{K}(t),\bar{G}(t)].
\]
Firstly, it follows from \eqref{GR1} that for $s\in[0,\tau]$,
\[
G(\tilde{s}_n+s)\leqslant X_*+\varepsilon+c^*(\tilde{s}_n+s)\leqslant
X_*+N\varepsilon+c^*(\tilde{s}_n+s)\leqslant \bar{G}(\tilde{s}_n+s).
\]
In view of \eqref{UR1}, we have
\begin{align*}
\bar{U}(\tilde{s}_n+s,y)&=\big(1+N\varepsilon e^{-\beta_0s}\big)\phi\big(y-\bar{G}(\tilde{s}_n+s)\big)\\
&\geqslant\big(1+N\varepsilon e^{-\beta_0\tau}\big)\phi\big(y-X_*-N\varepsilon-c^*(\tilde{s}_n+s)\big)\\
&\geqslant\phi\big(y-X_*-\varepsilon-c^*(\tilde{s}_n+s)\big)+\varepsilon\\
&\geqslant V\big(\tilde{s}_n+s,y-c^*(\tilde{s}_n+s)\big)=U(\tilde{s}_n+s,y ).
\end{align*}
for $s\in[0,\tau]$ and $y\leqslant G(\tilde{s}_n+s)$. By definition $\bar{U}(t,\bar{G}(t))=0$ and direct computation yields
\begin{eqnarray*}
-\mu \bar{U} _y(t,\bar{G}(t))& = & c^*\big(1+N\varepsilon e^{-\beta_0( t-\tilde{s}_n)}\big), \\
& < & c^*+N\varepsilon \sigma \beta_0 e^{-\beta_0 ( t-\tilde{s}_n)} = \bar{G}'(t),
\end{eqnarray*}
if we choose $\sigma$ with $\sigma\beta_0 > c^*$. Since $U\leqslant u^*$, it then follows from the definition of $\bar{K}(t)$ that
$\bar{U}(t,\bar{K}(t))=u^*\geqslant U(t,\bar{K}(t))$.

Finally, let us show
\begin{equation}\label{uhupper}
 \mathcal{N} [\bar{U}] := \bar{U}_t - \bar{U}_{yy} + d  \bar{U}-f(\bar{U}(t-\tau,y)) \geqslant 0,\quad y\in [\bar{K}(t), \bar{G}(t)],\ t>\tilde{s}_n+\tau.
\end{equation}
Put $z:=y-\bar{G}(t)$, $\zeta(t):=N\varepsilon e^{-\beta_0( t-\tilde{s}_n)}$ and $\phi_\tau:=\phi\big(y-\bar{G}(t-\tau)\big)$. It is easy to
compute that
\begin{align*}
\mathcal{N} [\bar{U}]&=\zeta\Big\{f(\phi_\tau)-\beta_0\phi-\sigma\beta_0(1+\zeta)\phi'-f'\big(\big(1+\theta_2\zeta e^{\beta_0\tau}\big)\phi_\tau\big)e^{\beta_0\tau}\phi_\tau\Big\}\ \ \ (\mbox{with } \theta_2\in(0,1))\\
&\geqslant\zeta\Big\{-\sigma\beta_0(1+\zeta)\phi'-\big[f'\big(\big(1+\theta_2\zeta e^{\beta_0\tau}\big)\phi_\tau\big)e^{\beta_0\tau}- d \big]\phi_\tau-\beta_0\phi\Big\}.
\end{align*}

Since
\[
\phi(z)\to u^*\ \mbox{ and } \frac{(\phi(z)-u^*)'}{\phi(z)-u^*}\to k^*\ \ \mbox{ as } z\to -\infty
\]
where $k^*:=c^*-\sqrt{(c^*)^2+4( d -f'(u^*))}<0$, there are two constants $z_\eta<0$  and $k_0$ such that
\begin{equation}\label{qqq122}
\phi''(z)>0,\ \ \ \phi(z)\geqslant u^*-\eta\ \ \mbox{ and }\ \ \phi'(z-2c^*\tau)
\geqslant  k_0 \phi'(z)  \ \mbox{ for } \ z<z_\eta,
\end{equation}
Moreover, we can compute that
\begin{align*}
\triangle \bar{G}(t) :& = \bar{G}(t)-\bar{G}(t-\tau) = c^*\tau+N\sigma \varepsilon e^{-\beta_0(t-\tilde{s}_n)}(e^{\beta_0\tau}-1).
\end{align*}
For any given $\sigma>0$, by shrinking $\varepsilon$ if necessary,  we have that
\begin{equation}\label{deh22}
\triangle \bar{G}(t)\in[c^*\tau,2c^*\tau]\ \ \mbox{ for }\ t > \tilde{s}_n+\tau.
\end{equation}

For $y-\bar{G}(t-\tau)\leqslant z_\eta$ and $t>\tilde{s}_n+\tau$, direct calculation implies
\begin{align*}
\mathcal{N} [\bar{U}]  &\geqslant\zeta\big\{-\sigma\beta_0(1+\zeta)\phi'-\big[f'\big(\big(1+\theta_2\zeta e^{\beta_0\tau}\big)\phi_\tau\big)e^{\beta_0\tau}- d \big]\phi_\tau-\beta_0\phi\big\}\\
 & \geqslant \zeta\big\{\big[d-f'\big(\big(1+\theta_2\zeta e^{\beta_0\tau}\big)\phi_\tau\big)e^{\beta_0\tau}-\beta_0\big]\phi_\tau-\sigma\beta_0\phi'+\beta_0(\phi_\tau-\phi)\big\}\\
&\geqslant \zeta\big[\beta_0 \phi'(y-\bar{G}(t)+\tilde{\theta}_2\triangle\bar{G}(t))\triangle\bar{G}(t)-\sigma\beta_0\phi'(y-\bar{G}(t))\big]\ \ \ (\mbox{with } \tilde{\theta}_2\in(0,1))\\
& \geqslant \zeta (2k_0c^*\tau-\sigma) \beta_0\phi'(y-\bar{G}(t)) \geqslant 0
\end{align*}
provided that $\sigma$ is sufficiently large, and we have used $\big(1+ \theta_2 \zeta e^{\beta_0\tau}\big)\phi_\tau\in[u^*-\eta,u^*+\eta]$
for $t>\tilde{s}_n+\tau$, \eqref{vuf1}, $\phi'(z)\leqslant0$ for $z\leqslant z_\eta$, \eqref{qqq122} and \eqref{deh22}.

When $z_\eta\leqslant y-\bar{G}(t-\tau)\leqslant 0$ and $t>\tilde{s}_n+\tau$, for sufficiently large $\sigma$, we have
\[
\mathcal{N} [\bar{U}]\geqslant\zeta \big[-\sigma\beta_0 C_z-u^*e^{\beta_0\tau}C_f -\beta_0u^*  \big]\geqslant0.
\]
where $C_z:=\max_{z\in[0,z_\eta+2c^*\tau]}\phi'(z)<0$, $C_f:=\max_{v\in[0,2u^*]}f'(v)$, and \eqref{deh22} are used.

Thus \eqref{uhupper} holds, then we can apply the comparison principle to conclude that
\[
U(t,y)\leqslant \bar{U}(t,y),\ \ \ G(t)\leqslant \bar{G}(t)\ \ \mbox{ for }\  y\in[\bar{K}(t),\bar{G}(t)]\ \mbox{and } t>\tilde{s}_n+\tau.
\]
This, together with the definition of $H(t)$, yields that
$H(t)\leqslant X_*+N\varepsilon(1+\sigma)$ for $t>\tilde{s}_n+\tau$.
By shrinking $\varepsilon$ if necessary, we obtain
\begin{equation}\label{Hsu1}
H(t)\leqslant X_*+\epsilon\ \ \ \mbox{ for }\ t>\tilde{s}_n+\tau\ \mbox{and } n>n_1.
\end{equation}

In the following, we show $H(t)\geqslant X^*-\epsilon$ for all large $t$. As in the construction of supersolution, for any $\varepsilon>0$,
there exists $n_2=n_2(\varepsilon)$ such that, for $n\geqslant n_2$,
\begin{eqnarray}
& H(s_n+s)\geqslant X^*-\varepsilon\ \ \mbox{ for } s\in[0,\tau], \label{GRH1}\\
& V(s_n+s,y)\geqslant \phi(y-X^*+\varepsilon)-\varepsilon\ \ \mbox{ for } s\in[0,\tau],\ y\leqslant X^*-\varepsilon. \label{URH1}
\end{eqnarray}
We also can find $N_0>1$ independent of $\varepsilon$ such that
\[
 \phi(y-X^*+\varepsilon)-\varepsilon \geqslant (1-N_0\varepsilon e^{-\beta_0\tau})\phi(y-X^*+N_0\varepsilon)\ \ \mbox{ for } y\leqslant X^*-\varepsilon,
\]
We can define a subsolution as follows:
$$
\begin{array}{l}
\underline{G}(t):= X^*-N_0\varepsilon+c^*t-N_0\sigma\varepsilon\big(1-e^{-\beta_0(t-s_n)}\big),\\
\underline{U}(t,y):=\big(1-N_0\varepsilon e^{-\beta_0(t-s_n)}\big)\phi\big(y-\underline{G}(t)\big).
\end{array}
$$
Since $U(t,y)\geqslant\phi(y-X_*)$, there are $C_0$ and $\alpha>0$ such that $V(t,y)\geqslant u^*-C_0e^{\alpha y}$
for all $y\leqslant0$, which implies that $U(t,y)\geqslant u^*-C_0e^{\alpha (y-c^*t)}$.
Let us fix $c\in(0,c^*)$ such that $\beta_0\leqslant \alpha(c+c^*)$. By enlarging $n$ if necessary we may assume that
$C_0\leqslant u^*N_0\varepsilon e^{\beta_0 s_n}$. Denote $\underline{K}(t)\equiv-ct$.

By a similar argument as above and in Step 2 of Proposition \ref{pro:sigma01}, we can show that
$(\underline{U},\underline{G},\underline{K})$ is a subsolution of problem \eqref{UGu} by taking  $\sigma>0$ sufficiently large.
The comparison principle can be used to conclude that
\[
\underline{U}(t,y)\leqslant U(t,y),\ \ \ \ \underline{G}\leqslant G(t)\ \ \ \mbox{ for } t\geqslant s_n+\tau,\ y\in[-ct,\underline{G}(t)],
\]
which implies that $G(t)\geqslant X^*-N_0\varepsilon(1+\sigma)$ for $t\geqslant s_n+\tau$.
By shrinking $\varepsilon$ if necessary, we have
\[
X^*-\epsilon\leqslant G(t)\ \ \ \mbox{ for } t\geqslant s_n+\tau\ \mbox{ and }\ n\geqslant n_2.
\]
This completes the proof of this lemma.
\end{proof}
\begin{thm}\label{thm:WHG}
Assume that {\bf (H)} and spreading happens.
Then there exists $H_1\in\R$  such that
\begin{equation}\label{HWt1}
\lim_{t\to\infty}[h(t) - c^*t ]= H_1,\ \ \ \ \ \lim_{t\to\infty} h'(t) =c^*,
\end{equation}
\begin{equation}\label{WHt1}
\lim\limits_{t\to\infty} \| u(t,\cdot)- q_{c^*}(c^*t+H_1-\cdot)\| _{L^\infty ( [0, h(t)])}=0,
\end{equation}
where $(c^*, q_{c^*})$ be the unique solution of \eqref{sw11}.
\end{thm}
\begin{proof}
It follows from Lemmas \ref{limn1} and \ref{limn21} that for any $t_n\to\infty$, by passing to a subsequence,
$h(t+t_n)-c^*(t+t_n)\to H_1:=H_0-2C$ in $C_{loc}^{1+\frac{\nu}{2}}(\R)$. The arbitrariness of $\{t_n\}$ implies
that $h(t)-c^*t\to H_1$ and $h'(t)\to c^*$ as $t\to\infty$, which proves \eqref{HWt1}.

In what follows, we use the moving coordinate $z:= x-h(t)$ to prove \eqref{WHt1}. Set
$$
g_2(t) := g(t)-h(t), \ \ \ \ \ \ u_2(t,z) := u(t, z+h(t))\ \ \mbox{ for }\ z\in [g_2 (t), 0],\ t\geqslant \tau,
$$
\[
\tilde{g}_n(t)=g(t+t_n)-h(t+t_n),\ \ \ \tilde{h}_n(t)=h(t+t_n),\ \ \ \ \tilde{u}_n(t,z)=u_2(t+t_n,z),
\]
then the pair $(\tilde{u}_n, \tilde{g}_n,\tilde{h}_n)$ solves
\begin{equation}\label{p u2}
\left\{
\begin{array}{ll}
 (\tilde{u}_n)_t =(\tilde{u}_n)_{zz}+\tilde{h}_n'(\tilde{u}_n)_z+ f(\tilde{u}_n(t-\tau,z+\tilde{h}_n(t)-\tilde{h}_n(t-\tau))- d  \tilde{u}_n, &  z\in(\tilde{g}_n(t),0),\ t>\tau,\\
 {\tilde{u}_n}(t, z)= 0,\ \ {\tilde{g}'_n}(t)=- \mu (\tilde{u}_n)_z (t,z)-\tilde{h}_n'(t), &  z={\tilde{g}_n}(t),\ t>\tau,\\
 {\tilde{u}_n}(t, 0)= 0,\ \ \tilde{h}_n'(t) = -\mu (\tilde{u}_n)_z(t,0), &  t>\tau.
 \end{array}
\right.
\end{equation}
By the same reasoning as in the proof of Lemma \ref{limn1}, the parabolic regularity to \eqref{p u2} plus the Sobolev
embedding theorem can be used to conclude that, by passing to a further subsequence if necessary, as $n\to\infty$,
$\tilde{u}_n\to W$  in $\ C_{loc}^{\frac{1+\nu}{2},1+\nu}(\R\times(-\infty,0])$, and $W$ satisfies, in view of $\tilde{h}'_n(t)\to c^*$,
\[
\left\{
\begin{array}{ll}
 W_t =W_{zz}+c^*W_z- d  W+ f(W(t-\tau,z+c^*\tau)), &   -\infty<z<0,\ t\in \R,\\
 W (t, 0)= 0, \ c^*=-\mu W_z (t,0), &  t\in \R.
\end{array}
\right.
\]
This is equivalent to \eqref{VGHQ} with $V=W$ and $H=0$. Hence we can conclude
\[
W(t,z)\equiv \phi(z)\ \ \ \mbox{ for } (t,z)\in\R\times(-\infty,0].
\]
Thus we have proved that, as $n\to\infty$,
\[
u(t+t_n,z+h(t+t_n))-q_{c^*}(-z)\to0 \ \ \   \mbox{ in } C_{loc}^{\frac{1+\nu}{2},1+\nu}(R\times(-\infty,0]).
\]
This, together with the arbitrariness of $\{t_n\}$, yields that
\[
\lim_{t\to\infty}[u(t,z+h(t))-q_{c^*}(-z)]=0\ \ \mbox{ uniformly for } z \mbox{ in compact subsets of } (-\infty,0].
\]
Then, for any $L>0$,
\[
\|u(t,\cdot)-q_{c^*}(h(t)-\cdot)\|_{L^\infty ([h(t)-L,h(t)])}\to0 \ \quad \mbox{ as } t\to \infty.
\]
Using the limit $h(t)-c^*t\to H_1$ as $t\to\infty$  we obtain
\begin{equation}\label{u to U near h(t)}
\|u(t, \cdot) - q_{c^*}(c^*t+H_1 -\cdot)\|_{L^\infty ([h(t)-L,h(t)])} \to 0\ \quad \mbox{ as } t\to \infty.
\end{equation}

Finally we prove \eqref{WHt1}. For any given small $\varepsilon >0$, it follows from \eqref{ughu1} in Proposition \ref{pro:sigma01}
that there exist two positive constants $Z_\varepsilon$ and $T_\varepsilon$ such that
$$
|u(t,x) - u^*| \leqslant u^*\varepsilon \ \quad \mbox{ for }\ 0\leqslant x\leqslant h(t) - Z_\varepsilon,\ t>T_\varepsilon.
$$
Since $q_{c^*}(z)\to u^*$ as $z\to\infty$, there exists $Z^*_\varepsilon > Z_\varepsilon$ such that
$$
|q_{c^*}(c^*t +H_1 -x) -u^*| \leqslant u^*\varepsilon\ \quad  \mbox{ for }\ x\leqslant c^*t+ 2H_1 -Z^*_\varepsilon.
$$
Taking $T^*_\varepsilon >T_\varepsilon$ large such that $h(t) <c^*t+2H_1$ for $t>T^*_\varepsilon$, then  we obtain
$$
|u(t,x) - q_{c^*}(c^*t +H_1 -x) | \leqslant 2u^*\varepsilon\ \quad \mbox{ for }\ 0\leqslant x\leqslant h(t)-Z^*_\varepsilon, \ t> T^*_\varepsilon.
$$
Taking $L= Z^*_\varepsilon$ in \eqref{u to U near h(t)} we see that for some $T^{**}_\varepsilon >T^*_\varepsilon$, we have
$$
|u(t, x) - q_{c^*}(c^*t +H_1 -x)| \leqslant u^*\varepsilon\  \quad \mbox{ for }\ h(t) -Z^*_\varepsilon \leqslant x \leqslant h(t),\ t>T^{**}_\varepsilon.
$$
This completes the proof of \eqref{WHt1}.
\end{proof}

Taking use of a similar argument as above one can obtain the following result.

\begin{thm}\label{thm:WGH}
Assume that {\bf (H)} and spreading happens. Then there exists $G_1\in\R$
such that
\begin{equation}\label{HWgt1}
\lim_{t\to\infty}[g(t) + c^*t ]= G_1,\ \ \ \ \ \lim_{t\to\infty} g'(t) =-c^*,
\end{equation}
\begin{equation}\label{WHtg1}
\lim\limits_{t\to\infty} \| u(t,\cdot)- q_{c^*}(c^*t-G_1+\cdot)\| _{L^\infty ( [g(t), 0])}=0,
\end{equation}
where $(c^*, q_{c^*})$ be the unique solution of \eqref{sw11}.
\end{thm}

\smallskip

\noindent
{\bf Proof of Theorem \ref{thm:profile of spreading sol}}. The results in Theorem \ref{thm:profile of spreading sol}
follow from Theorems \ref{thm:WHG} and \ref{thm:WGH}.
\hfill $\square$


\begin{thebibliography}{99}

\bibitem{AGT} M. Aguerrea, C. Gomez, S. Trofimchuk,
{\em On uniqueness of semi-wavefronts}, Math. Ann., 354 (2012), 73-109.


\bibitem{AW2}D.G.~Aronson and H.F. Weinberger,
{\em Multidimensional nonlinear diffusion arising in population genetics}, Adv. Math., 30 (1978), 33-76.

\bibitem{CF}
{ X.F. Chen and A. Friedman}, {\em A free boundary problem arising
in a model of wound healing}, SIAM J. Math. Anal., 32 (2000), 778-800.


\bibitem{BN92}
H. Berestycki and L. Nirenberg,
{\em Travelling fronts in cylinders},  Ann. Inst. H. Poincar\'e Anal. Non
Lin\'eaire 9(1992), 497-572.

\bibitem{CMYZ} I.-L. Chern, M. Mei, X. Yang Q. Zhang,
{\em Stability of non-monotone critical traveling waves for reaction-diffusion equations with time-delay},
J. Differential Equations, 259 (2015), 1503-1541.

\bibitem{CL} E. A. Coddington and N. Levinson,
{\it Theory of ordinary differential equations},
McGraw-Hill, New York, 1955.

\bibitem{BDK}
G.~Bunting, Y.~Du and K.~Krakowski,
{\em Spreading speed revisited: analysis of a free boundary model},
Netw. Heterog. Media,  7 (2012), 583-603.

\bibitem{DuGuo} Y.~Du, Z.M.~Guo,
{\em The Stefan problem for the Fisher-KPP equation}, J. Differential Equations, 253 (2012), 996-1035.

\bibitem{DGP} Y. Du, Z.M. Guo and R. Peng,
{\em A diffusion logistic model with a free boundary in time-periodic
environment}, J. Funct. Anal., 265 (2013), 2089-2142.

\bibitem{DuLin}
Y.~Du and Z.~Lin,
{\em Spreading-vanishing dichtomy in the diffusive logistic model
with a free boundary}, SIAM J. Math. Anal., 42 (2010), 377-405.

\bibitem{DuLou}
Y. Du and B. Lou,
{\em Spreading and vanishing in nonlinear diffusion problems with free boundaries}, J. Eur. Math. Soc., 17 (2015), 2673-2724.

\bibitem{DMZ}
Y. Du, H. Matsuzawa and M. Zhou,
{\em Sharp estimate of the spreading speed determined by nonlinear free boundary problems},
SIAM J. Math. Anal., 46 (2014), 375-396.

\bibitem{DMZ2} Y. Du, H. Matsuzawa, M. Zhou,
{\em Spreading speed and profile for nonlinear Stefan problems in high space dimensions}, J. Math. Pures Appl., 103  (2015), 741-787.

\bibitem{FangZhao14}
J. Fang and X.-Q. Zhao, {\em Traveling waves for monotone semiflows with weak compactness}, SIAM J. Math. Anal., 46 (2014),  3678-3704.

\bibitem{FangZhao15}
J. Fang and X.-Q. Zhao, {\em Bistable traveling waves for monotone semiflows with applications},
J. Eur. Math. Soc.(JEMS), 17 (2015),  2243-2288.

\bibitem{FM} {P.C. Fife and J.B.  McLeod, } {\em The approach of
solutions of nonlinear diffusion equations to travelling front
solutions}, { Arch. Ration. Mech. Anal.}, { 65} (1977), 335-361.

\bibitem{Fisher}
R.A.~Fisher,  {\em The wave of advance of advantageous genes}, Ann.
Eugenics, 7 (1937), 335-369.

\bibitem{Fr}
A. Friedman, {\it Parabolic Differential Equations of Parabolic Type}, Prentice-Hall Inc., Englewood Cliffs, N.J., 1964.


\bibitem{GBN}
W.S.C. Gurney, S.P. Blythe and R.M. Nisbet.,
{\em Nicholson's blowflies revisited}, Nature,
287 (1980), 17-21.

\bibitem{GT}
A. Gomez, S. Trofimchuk, {\em Global continuation of monotone wavefronts}, J. Lond. Math. Soc., 89 (2014), 47-68.

\bibitem{GW}S. A. Gourley and J. Wu,
{\em Delayed non-local diffusive systems in biological invasion and disease spread}, Nonlinear dynamics and evolution equations, 137-200, Fields Inst. Commun., 48, Amer. Math. Soc., Providence, RI, 2006.


\bibitem{KPP} A.N.~Kolmogorov, I.G.~Petrovski and N.S.~Piskunov,
{\em A study of the diffusion equation with increase in the amount
of substance, and its application to a biological problem}, Bull.
Moscow Univ. Math. Mech., 1 (1937), 1-25.


\bibitem{LZ}X. Liang and X.-Q. Zhao,
{\em Asymptotic speeds of spread and traveling waves for monotone
semifows with applications}, Comm. Pure Appl. Math., 60 (2007), 1-40.


\bibitem{LLLM}
C.-K. Lin, C.-T. Lin, Y. Lin and M. Mei,
{\em Exponential stability of nonmonotone traveling waves for Nicholson's blowflies equation},
SIAM J. Math. Anal., 46 (2014), 1053-1084.

\bibitem{LM}
C.-K. Lin and M. Mei,
{\em On travelling wavefronts of Nicholson's blowflies equation with diffusion},
Proc. Roy. Soc. Edinburgh Sect. A, 140 (2010), 135-152.

\bibitem{MS} R.H. Martin, H.L. Smith,
{\em Abstract functional differential equations and reaction-diffusion systems}, Trans. Amer. Math. Soc.,
321 (1990) 1-44.

\bibitem{MLLS} M. Mei, C.-K. Lin, C.-T. Lin and J.W.-H. So,
{\em Traveling wavefronts for time-delayed reaction-diffusion equation. I. Local nonlinearity}, J. Differential Equations, 247 (2009), 495-510.

\bibitem{SWZ} J. W.-H. So, J. Wu and X. Zou,
{\em A reaction-diffusion model for a single species with age structure. I. Travelling wavefronts on unbounded domains.} Proc. Roy. Soc. Lond.  A., 457 (2001), 1841-1853.


\bibitem{MSLS}
M. Mei, J.W.-H. So, M.Y. Li and S.S.P. Shen,
{\em Asymptotic behavior of traveling waves for the  Nicholson's blowflies equation with diffusion},
Proc. Roy. Soc. Edinburgh Sect. A, 134 (2004), 579-594.

\bibitem{N} A.J. Nicholson,
{\em An outline of the dynamics of animal populations}, Aust. J. Zool. 2 (1954),
9-65.


\bibitem{RuanWei2003}S. Ruan and J. Wei, On the zeros of transcendental functions with
applications to stability of delay differential equations with two delays,
{\em Dynamics of Continuous, Discrete and Impulsive Systems}, 10(2003), 863-874.


\bibitem{Sc} K. Schaaf,
{\em Asymptotic behavior and traveling wave solutions for parabolic functional
differential equations}, Trans. Amer. Math. Soc. 302 (1987), 587-615.

\bibitem{SZ} H.L. Smith and X.-Q. Zhao,
{\em  Global asymptotic stability of traveling waves in delayed
reaction-diffusion equations}, SIAM J. Math. Anal., 31 (2000), 514-534.

\bibitem{SY} J.W.-H. So and Y. Yang,
{\em Dirichlet problem for the diffusive Nicholson's blowflies equation},
J. Differential Equations, 150 (1998), 317-348.

\bibitem{TZ} H.R. Thieme, X.-Q. Zhao,
{\em Asymptotic speeds of spread and traveling waves for integral equations and delayed reaction-diffusion models},
J. Differential Equations, 195 (2003), 430-470.


\bibitem{W} M. Wang,
{\em Existence and uniqueness of solutions of free boundary problems in heterogeneous environments},
Discrete Contin. Dyn. Syst.-B. DOI: 10.3934/
dcdsb.2018179.

\bibitem{WLR}Z.-C. Wang, W.-T. Li and S. Ruan,
{\em Traveling fronts in monostable equations with nonlocal
delayed effects}, J. Dyn. Diff. Equat., 20 (2008), 573-607.


\bibitem{WZ} J. Wu, X. Zou,
{\em Traveling wave fronts of reaction-diffusion systems with delay}, J. Dyn. Diff. Equat., 13
(2001), 651-687.

\bibitem{YCW} T. Yi, Y. Chen, J. Wu,
{\em Threshold dynamics of a delayed reaction diffusion equation subject to the Dirichlet condition}, J. Biol. Dyn., 3 (2009), 331-341.


\bibitem{ZW} X. Zou and J. Wu,
{\em Existence of traveling wave fronts in delayed reaction-diffusion systems via the monotone iteration method},
Proc. Amer. Math. Soc., 125 (1997), 2589-2598.

\end{thebibliography}
\end{document}